\documentclass[11pt,reqno,a4paper]{amsart}
\usepackage[centertags]{amsmath}
\usepackage{amsfonts}
\usepackage{amssymb}
\usepackage{amsthm}
\usepackage{times}
\usepackage[top=2.5cm, bottom=2.5cm, left=2.5cm, right=2.5cm]{geometry}
\usepackage{color}
\usepackage{graphicx}
\usepackage{hyperref}
\hypersetup{
    colorlinks=true,
    linkcolor=blue,
    filecolor=magenta,
    urlcolor=cyan,
}

\numberwithin{equation}{section}

\newcommand{\cG}{\mathcal{G}_{\Phi^o}}
\newcommand{\cH}{\mathcal{H}}
\newcommand{\cK}{\mathcal{A}_c}
\newcommand{\cM}{{\mathcal M}}

\newcommand{\N}{\mathbb{N}}
\newcommand{\R}{\mathbb{R}}
\renewcommand{\S}{\mathbb{S}}

\newcommand{\doublevariable}{z}

\newcommand{\An}{\widehat A}
\newcommand{\Bn}{\widehat B}
\newcommand{\En}{\widehat E}
\renewcommand{\epsilon}{\varepsilon}

\newcommand{\Fn}{\widehat F}

\newcommand{\sg}{\mathrm{sg}}
\newcommand{\oic} {[0,+\infty)}

\newcommand{\opensetRn}{\widehat \Omega}
\newcommand{\opensetRnplusone}{\Omega}
\newcommand{\bopensetRn}{\An}
\newcommand{\bopensetRnplusone}{A}
\newcommand{\genericopenset}{O}

\newcommand{\hsectangent}{{\widehat \xi}}
\newcommand{\hsecdual}{{\widehat\xi}^*}

\newcommand{\direction}{\zeta}
\newcommand{\level}{\lambda}

\newcommand{\loc}{\mathrm{loc}}
\newcommand{\Rn}{\mathbb{R}^n}

\newcommand{\norm}{\Psi}

\newcommand{\phisec} {\varphi}

\newcommand{\strictlyincluded}{\subset\subset}

\newcommand{\test}{\psi}

\newcommand{\unitballPhi} {B_\Phi}
\newcommand{\unitballPsi} {B_\Psi}

\newcommand{\supp}{{\ensuremath{\mathrm{supp\,}}}}

\newcommand{\p}{\partial}
\newcommand{\esssup}{\mathop{\mathrm{ess\, sup}}}

\renewcommand{\div}{\mathop{\mathrm{div}}}
\newcommand{\essinf}{\mathop{\mathrm{ess\, inf}}}

\newcommand{\sign}{\mathop{\mathrm{sign}}}

\theoremstyle{plain}
\newtheorem{theorem}{Theorem}[section]
\newtheorem{lemma}[theorem]{Lemma}

\newtheorem{proposition}[theorem]{Proposition}
\newtheorem{corollary}[theorem]{Corollary}

\theoremstyle{definition}
\newtheorem{example}[theorem]{Example}
\newtheorem{definition}[theorem]{Definition}

\newtheorem{remark}[theorem]{Remark}

\date{\today}

\begin{document}

\title[Minimizers of anisotropic perimeters with cylindrical norms]
{Minimizers of anisotropic perimeters with cylindrical norms}

\author[G. Bellettini, M. Novaga, Sh. Yu. Kholmatov]{G. Bellettini$\!^1$,
 M. Novaga$\!^2$, Sh. Yu. Kholmatov$\!^{3,4}$}

\begin{abstract}
We study various  regularity properties of 
minimizers of the $\Phi$--perimeter, where $\Phi$ is a norm. 
 Under suitable assumptions on $\Phi$ and on
the dimension of the ambient space, we prove that the boundary of a 
cartesian minimizer  is locally a Lipschitz graph out of a
closed singular set of small Hausdorff dimension. Moreover,
we show the following anisotropic Bernstein-type result:
any entire cartesian minimizer 
is the subgraph of a monotone function depending only on one variable.
\end{abstract}


\address{$^1\!$Dipartimento di Matematica, Universit\`a di 
Roma ``Tor Vergata'',
Via della Ricerca Scientifica 1, 00133 Roma, Italy}%

\email{$^1\!$belletti@mat.uniroma2.it}

\address{$^2\!$Dipartimento di Matematica, Universit\`a di Pisa,
Largo Bruno Pontecorvo 5, 56127 Pisa, Italy}

\email{$^2\!$novaga@dm.unipi.it}

\address{$^3\!$International Centre for Theoretical Physics (ICTP),
Viale Miramare 11, 34156 Trieste, Italy}

\address{$\!^4\!$Scuola Internazionale Superiore di
Studi Avanzati (SISSA),
Via Bonomea 265, 34136 Trieste, Italy}

\email{$^{3,4}\!$sholmat@sissa.it}

\maketitle

\tableofcontents

\section{Introduction}\label{sec:introduction}
In this paper we are interested in regularity
properties of  minimizers
of the anisotropic perimeter
$$
P_\Phi(E,\Omega) = \int_{\Omega\cap \partial^* E} \Phi^o(\nu_E)~d\mathcal H^{n},
$$
of $E$ in $\Omega,$ and of the related area-type functional
$$
\cG(v, \opensetRn) = \int_{\opensetRn} \Phi^o(-Dv,1).
$$
 Here $\Omega \subseteq \R^{n+1}$  is an open set,
$\Phi: \R^{n+1} \to [0,+\infty)$ is a norm
 (called anisotropy), $\Phi^o$ is its dual, $E\subset \R^{n+1}$ is a set of locally finite perimeter,
$\partial^*E$ is its reduced boundary,
$\nu_E$ is the outward (generalized) unit normal to $\partial^* E,$
and $\mathcal H^{n}$ is the $n$-dimensional
Hausdorff measure in $\R^{n+1}.$ On the other hand, $\widehat \Omega \subseteq \Rn,$  
$v$ belongs to
the space   $BV_{\rm loc}(\widehat \Omega)$ of functions with locally bounded total variation in 
$\widehat \Omega,$ and $Dv$ is the 
distributional derivative of $v.$  When $\Omega=\widehat\Omega\times\R$ the two functionals
coincide provided $E$ is cartesian, {\it i.e.} 
 $E$ is the subgraph $\sg(v) \subset \widehat \Omega
\times \R$ of the function $v\in BV_\loc(\widehat \Omega)$
(see \eqref{minP}).

Anisotropic perimeters appear in many models in material science 
and phase transitions \cite{Gur:93, Tay:78},
in crystal growth \cite{BNP:2001, BCCN:05, CH:1972, CH:1974, TC:86, AT:1995},
and in boundary detection
and tracking \cite{CKSS:97}.
Functionals like $\mathcal G_{\Phi^o},$ 
having  linear growth in the gradient, 
 appear quite frequently
in calculus of variations \cite{Gi:84, BPV:96, AFP:00}. The one-homogeneous case is particularly
relevant, since it is related to the 
{\it anisotropic total variation functional}
\begin{equation}\label{eq:TVphisec}
TV_{\phisec}(v,\opensetRn) = \int_{\opensetRn} \phisec^o(Dv),
\end{equation}
a useful functional appearing, for example,
 in  image reconstruction and
denoising \cite{ROF92, Cham:04, Cham:09, AmBe:, Moll:05}. Here $\phisec: \Rn \to [0,+\infty)$ 
is a norm, and its dual $\phisec^o$ is typically the restriction 
of $\Phi^o$ on the ``horizontal'' $\R^n.$

Minimizers  of $P_\Phi$ have been widely studied 
\cite{Tay:78,ATW:95}; in particular, it is known 
\cite{Bom:82,AlScYa:77} 
that  if $\Phi^2$ is smooth and uniformly convex, 
(boundaries of) minimizers are smooth out of a ``small'' closed singular set.
In contrast to the classical case, where perimeter minimizers are 
smooth  out of
a closed set of Hausdorff dimension at most $n-7,$ 
the behaviour of minimizers of anisotropic perimeters is more irregular: for instance, there exist 
singular minimizing cones even for smooth and uniformly convex anisotropies in  $\R^4$
\cite{Mo:}.  Referring to functionals of the form \eqref{eq:TVphisec}, 
we recall that, if
$n\le 7,$  H\"older continuity of minimizers 
for the 
image denoising functional  \cite{ROF92}, consisting of the Euclidean total
variation $TV$ plus the usual quadratic
fidelity term, has been studied in \cite{VChN:11}.
In \cite{Me:15} such result is extended to the anisotropic total variation $TV_\phisec$.

One of the remarkable results in the classical theory of minimal surfaces is the classification of 
entire minimizers of the Euclidean  perimeter $P$: 
if $n\le 6$ the only entire minimizers 
are hyperplanes, while for $n=7$ there are nonlinear entire
minimizers (see for instance \cite[Chapter 17]{Gi:84} and references therein);
in the cartesian case (sometimes called the non parametric case),
this is the  well-known Bernstein problem. 
In the anisotropic setting, to our best knowledge, only a few results are available:
entire minimizers in $\R^2$ are classified  in \cite{NoPa:05}, and 
minimizing cones in $\R^3$ for crystalline 
anisotropies are classified in \cite{Tay:86}. 
In \cite{Je:61,Si:77} the authors show that if $n\le 2$ and $\Phi^2$ is smooth,
the only 
 entire cartesian minimizers are the subgraphs of linear functions
(anisotropic Bernstein problem), and the same result holds up to dimension 
$n\le 6$ if $\Phi$ is close enough to the Euclidean norm \cite{Si:77}.
However, the anisotropic Bernstein problem seems to be still open
in dimensions $4\le n\le 6,$ even for smooth and uniformly convex norms
(see \cite{Ovvo:14} for recent results in this direction).

The above  discussion shows the difficulty of describing perimeter
minimizers in the presence of an anisotropy; it seems therefore rather natural
to look for reasonable assumptions on $\Phi$ that allow
to simplify the classification problem. A possible
requirement, which will be often (but not always)
assumed in the sequel of the paper,
is that $\Phi$ is cylindrical over $\phisec,$ {\it i.e.}
\begin{equation}\label{eq:cyl}
\Phi(\hsectangent,\xi_{n+1}) = \max\{\phisec(\hsectangent),|\xi_{n+1}|\},
\quad (\hsectangent,\xi_{n+1})\in\R^{n+1}.
\end{equation}
Despite its splitted expression, 
a cylindrical anisotropy  is neither smooth nor
strictly convex, and this still makes the above mentioned 
classification rather complicated. 
For instance, 
in Examples \ref{ex:minEx} and \ref{ex:Union of cones} 
we show that there  exist 
singular cones minimizing $P_\Phi$ in any dimension $n\ge1.$
Moreover, while 
it can be proved that if
horizontal and vertical sections of $E$ 
are minimizers of $P_\phisec$ and $P$ 
respectively
then $E$ is a minimizer of $P_\Phi$  (Remark \ref{rem:minimal_sections}),
in general sections of a minimizer of $P_\Phi$ 
need not satisfy   this minimality property
(Examples \ref{ex:2dim_strip} and \ref{ex:Union of cones}). 

These phenomena lead us to investigate the classification problem under 
some  simplifying assumptions on the structure of  minimizers. We shall consider
two cases: cylindrical minimizers
(Definition \ref{def:cylindric_min}), and cartesian minimizers
(Definition \ref{def:cartesian_min}), the latter
being our main interest.
Cylindrical minimizers of $P_\Phi$ are studied in 
Section \ref{sec:cylindrical_minimizers}:
in particular, in Example \ref{ex:char.cylind.min} 
we classify all cylindrical minimizers of $P_\Phi$ when $n=2$ and the unit ball $B_\Phi$ 
of $\Phi$ (sometimes called Wulff shape) is a 
cube. Cartesian minimizers are studied in Sections 
\ref{sec:cartesian_minimizers}, 
\ref{sec:cartesian_minimizers_for_cylindrical_norms} and \ref{sec: LipRegCarMin}.
In Section
\ref{sec:cartesian_minimizers} we investigate
the relationships between cartesian minimizers of $P_\Phi$ 
and minimizers of $\cG,$
provided $\Phi$ is partially monotone (Definition \ref{def:partial_monotone_norm}). 
In Theorem \ref{teo:Subgraphs} we show that 
the subgraph of a minimizer of $\cG$ 
is also a minimizer of $P_\Phi$  among all perturbations not preserving the 
cartesian structure.
In particular, for $\Phi$ satisfying \eqref{eq:cyl} the subgraph $E$ of some 
function $u:\opensetRn\to\R$
is a cartesian minimizer of $P_\Phi$ in $\opensetRn\times\R$ 
if and only if $u$ is a minimizer of 
$TV_\phisec.$

Sections 
\ref{sec:cartesian_minimizers_for_cylindrical_norms} and \ref{sec: LipRegCarMin} contain
our main results,
 valid under the assumptions that  
$$
\Phi \,\,\text{is  cylindrical  over   $\phisec$ and }\,\, E  {\rm  ~is~ cartesian}.
$$
In Theorem \ref{teo:Euc} (see also Corollary \ref{cor:char.min.cylindrical}) 
we prove the following 
Bernstein-type classification result: {\it if 
eiter $n\le 7$ and $\phisec$ is Euclidean, or if $n=2$ and 
$\phisec^o$ is strictly convex,  
then any entire cartesian  minimizer of $P_\Phi$ in $\R^{n+1}$
({\it i.e.} the subgraph of a minimizer of 
$TV_\phisec$)  is the subgraph of the 
composition of a monotone function on $\R$ with a linear function on $\R^n.$}
We notice that this result is sharp: if $n=8,$ there are entire cartesian minimizers
of $P$ in $\R^9$ which cannot be represented as the subgraph of the
composition of a  monotone and a linear function (see Remark \ref{rem:opt_Euc}).

In view of our assumptions, 
also the regularity results of Section \ref{sec: LipRegCarMin} are concerned with
 the anisotropic total variation 
functional. For our purposes, 
it is useful
to  remark that, even if the anisotropy $\phisec$ 
is smooth and uniformly convex, in general minimizers of $TV_\phisec$
are not necessarily continuous. In contrast, we remark that minimizers of $TV_\phisec$
with continuous boundary data on bounded domains are continuous, 
see \cite{Je:15, Me:15, Ma:16}. 
Nevertheless, in Theorems \ref{teo: reg_in_all_dim} and \ref{teo:Al} 
we show that, {\it if 
$\phisec^2\in C^3$ is uniformly convex,  then the boundary of the subgraph 
of a minimizer of $TV_\phisec$ is locally Lipschitz (that is, locally a Lipschitz graph)
out of a closed singular set with a suitable 
 Hausdorff dimension depending on $\phisec$}. 
As observed in 
Remark \ref{rem:opt_Bernstein}, 
for $\phisec$ Euclidean these statements  are optimal,
while the statement is false already in dimension $n=2$ for $\phisec$ the square norm.

\section{Notation and preliminaries}\label{sec:notation_and_preliminaries}

\jot=10pt

In what follows $n \geq 1$,
$\opensetRnplusone\subseteq \R^{n+1}$ and 
$\opensetRn \subseteq \Rn$ are open sets.
$BV(\opensetRnplusone)$ (resp. $BV_{{\rm loc}}(\opensetRnplusone)$)
stands for the space of functions with bounded (resp. locally bounded)
variation in $\opensetRnplusone$ \cite{AFP:00}.
The characteristic function of a (measurable) set $E\subset\Omega$ is denoted by
$\chi_E;$ we write $E \in BV(\opensetRnplusone)$ (resp. $E \in 
BV_{{\rm loc}}(\opensetRnplusone)$) when $\chi_E \in BV(\opensetRnplusone)$ (resp.
$\chi_E \in BV_{{\rm loc}}(\opensetRnplusone)$). Similar notation holds
in $\opensetRn.$
$P(E,A)$ denotes the Euclidean perimeter of the set $E$ in the open 
set $A.$ Recall  that the  perimeter of $E\in BV_\loc(\opensetRnplusone)$ does not change if
we change $E$ into another set in the same Lebesgue equivalence class;
henceforth we shall always assume
that any set $E$ coincides with its points of density one \cite{AFP:00, Fed:69}.
The outward generalized unit normal to the reduced boundary $\partial^*E$ of $E \in 
BV_{{\rm loc}}(\Omega)$ is denoted by $\nu_E.$
We often use the splitting
$\R^{n+1} = \{(x,t):\,\,x\in\R^n,\,t\in\R\}$
and write $\nu_E = (\widehat \nu_E, (\nu_E)_t)$ and 
$e_{n+1} = (0,\dots,0,1).$
If  $F\subseteq \R^{n+1},$  $x\in\R^n,$ $t\in\R$ we let
\begin{equation}\label{eq:Sections}
F_t:=\{y\in\R^n:\quad (y,t)\in F\},\quad 
F_x:=\{s\in\R:\quad (x,s)\in F\}. 
\end{equation}
Unless otherwise specified, in the sequel
we take $m\in\{n,n+1\}$.

$\cK(\opensetRnplusone)$ is the collection of all open
relatively compact subsets of $\opensetRnplusone.$  
The sequence $\{E_h\}$ of 
 subsets of $\R^m$ 
converges to set $E\subset\R^m$ in $L_\loc^1(\opensetRnplusone)$ if
$\chi_{E_h}\to\chi_E$ as $h\to+\infty$ in $L^1(A)$ 
for any $A\in\cK(\opensetRnplusone).$

Finally,
for a function  $u : \opensetRn \to \R$ we let
$$\sg(u):=\{(x,t)\in\R^{n+1}:\,\,x\in\opensetRn,\,\,u(x) > t\}$$
be the subgraph of $u.$
We recall \cite{Gi:84, MM1964}
that
\begin{equation}\label{impassump}
P_\Phi(\sg(u), \An\times \R) <+\infty \qquad \forall \An\in\cK(\opensetRn).
\end{equation}
if and only if $u\in BV_{\loc}(\opensetRn).$ 
%

\subsection{Norms}\label{subsec:norms}
A norm on $\R^m$ 
is a convex function $\norm : \R^m \to [0,+\infty)$ satisfying 
$\norm(\lambda \xi) = |\lambda|\norm(\xi)$ for all $\lambda>0$ and $\xi\in\R^m,$
and for which 
there exists a constant $c>0$ such that
\begin{equation}\label{comparison}
 c |\xi| \le \norm(\xi), \qquad 
\xi\in\R^m.
\end{equation}

We let
$\unitballPsi := \{\xi \in \R^m : \norm(\xi)\leq 1\},$
which is sometimes called Wulff shape,
and   $\norm^o: (\R^m)^* \to \oic$ the dual norm of $\norm,$ 
$$\norm^o(\xi^*) = \sup\{ \xi^*\cdot\xi:\,\,\xi\in B_\norm\},
\qquad 
\xi^* \in (\R^m)^*,
$$
where $(\R^m)^*$ is the dual of $\R^m,$ and $\cdot$ is the Euclidean
scalar product.
We have
\begin{equation}\label{duality}
\xi^*\cdot \xi \le \norm^o(\xi^*)\norm(\xi),\quad \xi^*\in(\R^m)^*,\,\,\xi\in\R^m,
\end{equation}
and $\Psi^{oo}=\Psi.$
When $m=n+1$ we often split $\xi \in \R^{n+1}$ as $\xi = (\hsectangent, \xi_{n+1}) \in \Rn \times
\R,$ and employ the symbol $\Phi$ (resp. $\phisec$) to denote a norm
in $\R^{n+1}$ (resp. in $\Rn$). In $\R^{n+1}$ we
frequently exploit the restriction
$\Phi_{\vert_{\{\xi_{n+1} =0\}}}^{}$ of  $\Phi$ 
to the horizontal hyperplane $\{\xi_{n+1}=0\},$ 
which is a norm on $\R^n.$
Note that 
\begin{equation}\label{restHorPl}
\left(\Phi_{\vert_{\{\xi_{n+1} =0\}} }^{} \right)^o 
\le \Phi_{\vert_{\{\xi_{n+1}^* =0\}}}^o.
\end{equation}
Indeed, let 
$$\phisec:=\Phi_{ \vert_{\{\xi_{n+1} =0\}} }^{} \,\,\, \text{and}\,\,\, 
\phi:=\left(\Phi_{ \vert_{ \{\xi_{n+1}^* =0\} } }^o\right)^o.$$
Fix $\hsecdual\in\R^n$ and choose $\hsectangent\in\R^n$ such that 
$\phisec(\hsectangent)=\Phi(\hsectangent,0)=1$ and
$\phisec^o(\hsecdual)=\hsectangent\cdot \hsecdual.$ Thus,
$$\phisec^o(\hsecdual)=(\hsectangent,0)\cdot 
(\hsecdual,0)\le \Phi^o(\hsecdual,0)=\phi^o(\widehat{\xi^*}). $$ 

\begin{remark}\label{rem:StrInq}
Inequality  \eqref{restHorPl} may be strict. 
For 
$\alpha\in(0,\pi/2)$ 
consider
the symmetric parallelogram with vertices 
at $(1\pm\cot\alpha,\pm1),$ $(-1\pm\cot\alpha,\pm1),$ and let
 $\Phi_\alpha$ be the Minkowski
functional of $P_\alpha.$
Notice that
$$(\Phi_\alpha)_{\vert_{\{\xi_2 =0\}} }^{}(\xi_1) = |\xi_1|$$ and
$$  (\Phi_\alpha^o)_{\vert_{\{\xi_2^* =0\}}}^{}(1) = \Phi_\alpha^o(1,0) = 
\sup\{\xi_1:\,\,(\xi_1,\xi_2)\in P_\alpha\} = 1+\cot\alpha,$$
thus
$$
\left((\Phi_\alpha)_{\vert_{\{\xi_2 =0\}}}^{}\right)^o(1) = 1<
1+\cot\alpha= (\Phi_\alpha^o)_{\vert_{\xi_2^* =0}}^{}(1).
$$
In Lemma \ref{lem:CommutRestAdj} we give necessary and 
sufficient conditions on $\Phi$ 
ensuring that equality 
in \eqref{restHorPl} holds. 
\end{remark}

\begin{definition}[\textbf{Cylindrical and conical norms}]\label{def:cylindrical_norm}
We say that the norm $\Phi: \R^{n+1} \to [0,+\infty)$
is {\it cylindrical} over $\phisec$ if
\begin{equation}\label{eq:phi_dual_conical}
\Phi(\hsectangent,\xi_{n+1}) = \max\{\phisec( \hsectangent ),|\xi_{n+1}|\},
\qquad (\hsectangent,\xi_{n+1}) \in \R^{n+1},
\end{equation}
where $\phisec: \R^n \to \oic$ is a  norm.
We say that $\Phi: \R^{n+1} \to \oic$
is {\it conical} over $\phisec,$ if 
\begin{equation*}
\Phi(\xi) = \phisec(\hsectangent) + \vert \xi_{n+1}\vert, 
\qquad (\hsectangent,\xi_{n+1}) \in \R^{n+1}.
\end{equation*}
\end{definition}
Notice that if $\Phi$ is cylindrical over $\phisec$ then
$\Phi^o$ is conical over $\phisec^o,$ and vice-versa.

\subsection{Perimeters}\label{subsec:perimeter}
Let $\norm: \R^m \to \oic$ be a norm  and
 $\genericopenset \subseteq \R^m$
be an open set. For any $E\in BV_{{\rm loc}}(\genericopenset)$ 
and for any $\bopensetRnplusone \in \cK(\genericopenset)$ we define \cite{AmBe:}
 the $\norm$-perimeter of $E$ in $\bopensetRnplusone$ as
$$P_\norm(E,\bopensetRnplusone):=\int_{\bopensetRnplusone}
\norm^o(D\chi_E) =
\sup\left\{-\int_E\div \eta \,dx:\,\, \eta\in C_c^1(\bopensetRnplusone, B_\norm)\right\}.
$$
It is known \cite{AmBe:} that
\begin{equation}\label{perii}
P_\norm(E,\bopensetRnplusone)=
\int_{\bopensetRnplusone\cap \partial^*E}
\norm^o(\nu_E)\,d\cH^{m-1}.
\end{equation}

\begin{definition}[\textbf{Minimizer of anisotropic perimeter}]\label{def:local_minimizer_of_F}
 We say that $E\in BV_{{\rm loc}}(\genericopenset)$
is a {\it minimizer}
 of $P_\norm$ {\it  by compact perturbations
in} $\genericopenset$ (briefly, a {\it minimizer of} $P_\norm$ {\it in} $\genericopenset)$
if
\begin{equation}\label{eq:local_minimizer}
P_\norm(E,\bopensetRnplusone)\le P_\norm(F,\bopensetRnplusone)
\end{equation}
for any $\bopensetRnplusone\in\cK(\genericopenset)$
and $F\in BV_{{\rm loc}}(\genericopenset)$
such that
$E\Delta F\strictlyincluded\bopensetRnplusone.$
\end{definition}

From \eqref{perii} it follows that if  $E$ is minimizer
of $P_\norm$ in $\genericopenset,$ then   so is $\R^m\setminus E.$ If $m=1,$ then 
$\Phi(\xi) = \Phi(1)|\xi|,$ 
thus $E\subset\R$ is a minimizer of $P_\norm$ in an open interval $I$
if and only if it is a minimizer of the Euclidean perimeter,
so $E$ is of the form
\begin{equation}\label{eq:1_dimensional_minimizers}
\emptyset,\quad I,\quad (-\infty,\lambda)\cap I,
\quad(\lambda,+\infty)\cap I, \qquad \lambda\in I. 
\end{equation}

The following example is  
based on a standard calibration argument\footnote{See 
for instance \cite{ABD03} for some definitions, results and references
 concerning
calibrations.}.

\begin{example}[\textbf{Half-spaces}]\label{ex:half-spaces}
\rm
 Let $H\subset\R^m$ be a half-space and $\genericopenset \subseteq \R^m$
be open.
Then $E=H\cap \genericopenset$ 
is a minimizer of $P_\norm$ in $\genericopenset.$
Indeed, let $\direction\in\R^m$ be
such that $\norm(\direction)=1$ and $\nu_H^{}\cdot \direction =\norm^o(\nu_H^{}).$
Consider $F \in BV_{\rm loc}(O)$
with  $E\Delta F\strictlyincluded A\strictlyincluded\genericopenset.$
Observe that
$\partial^*(E\setminus F)$ can be written as a pairwise 
disjoint\footnote{Up to sets of 
zero $\cH^{m-1}$-measure.}
 union 
of  $(\R^m\setminus F) \cap \partial E,$
$E \cap \partial^* F$ and  $J := \{\doublevariable \in \partial E 
\cap \partial^* F : \nu_H^{}(\doublevariable)
 = -\nu_F^{}(\doublevariable)\}$ (see for example \cite[Theorem 16.3]{Maggi}).
For the vector field
$N:\R^m\to \R^m$ constantly equal to $\direction,$
we have
\begin{equation}\label{equai1}
\begin{aligned}
0=&\int_{E\setminus F}
\div N ~d\doublevariable
\\
=& \int_{(\R^m\setminus F)\cap A \cap \p E}
\nu_H^{}\cdot N\,d\cH^{m-1} -\int_{E
\cap A \cap \p^*  F}
\nu_F^{}\cdot N\,d\cH^{m-1} + \int_{J\cap A} \nu_H^{}\cdot N\,d\cH^{m-1}\\
=:& {\rm I}-{\rm II} + {\rm III}.
\end{aligned} 
\end{equation}
Similarly,
\begin{equation}\label{equai2}
\begin{aligned}
0=&\int_{F\setminus E} \div (-N) ~d\doublevariable
\\
=& -\int_{
(\R^m\setminus E)\cap A \cap  \partial^* F}
\nu_F^{}\cdot N\,d\cH^{m-1} +\int_{F \cap A \cap \partial E}
\nu_H^{}\cdot N\,d\cH^{m-1} - \int_{J\cap A} \nu_F^{}\cdot N\,d\cH^{m-1}\\
=:& -{\rm IV} +{\rm V}- {\rm VI}.
\end{aligned}
\end{equation}
Adding \eqref{equai1}-\eqref{equai2} and using $\nu_F^{} \cdot N \leq \norm^o(\nu_F^{})$ we obtain 
\begin{align*}
& \int_{A \cap \partial E} \norm^o(\nu_H)~d\mathcal H^{m-1} = 
\int_{A \cap \partial E} \nu_H^{} \cdot N~d\mathcal H^{m-1}
= {\rm I}+ {\rm III} +{\rm V} 
\\
=& {\rm II}+ {\rm IV} +{\rm VI}= \int_{A \cap \partial^* F} \nu_F^{} \cdot N~d\mathcal H^{m-1}
\le  \int_{A \cap \partial^* F} \norm^o(\nu_F^{})~d\mathcal H^{m-1}.
\end{align*}
\end{example}
\bigskip

The previous argument does not apply to a strip between two parallel planes.

\begin{example}[\textbf{Parallel planes}]\label{ex:parallel_planes}\rm
Let $n=2,$ and let $\Phi:\R^3\to\oic$ be cylindrical over the Euclidean
norm.
Given $a<b$ consider
$E = \{(x,t) \in \R^3:\, a < t < b\}.$
Then $E$ is not a
minimizer of $P_\Phi$
in $\R^3.$ Indeed, it is sufficient to compare $E$ with the set
$E \setminus C,$
obtained from $E$ by removing a  sufficiently large cylinder
$C = B_R \times [a,b]$ homothetic
to $\unitballPhi,$ where $B_R  = \{x\in \R^2 :\, \vert x\vert <R\}.$
Then $P_\Phi(E)$ is reduced by $2 \pi R^2$ (the sum of the 
areas of the top and bottom
facets of $C$), while it is increased by the lateral
area $2 \pi(b-a)R$
of $C.$
Hence, for $R>0$ sufficiently large,
\eqref{eq:local_minimizer} is not satisfied.
Notice that the horizontal sections of $E$ are
either empty or
a plane, which both are  minimizers of the Euclidean perimeter
in $\R^2.$
\end{example}

\begin{remark}\label{rem:minimal_sections}
Suppose that   $\Phi:\R^{n+1}\to\oic$ 
is cylindrical over $\phisec.$
Assume that $E\in BV_\loc(\opensetRnplusone)$ has the following property:
for almost every  $t\in\R$ the set $E_t$
(horizontal section)
is a  minimizer of $P_{\phisec}$
in $\opensetRnplusone_t$ and for almost every
 $x\in\R^n$ the set $E_x$
(vertical section)
is a minimizer of Euclidean perimeter
in $\opensetRnplusone_x.$
Then  by Remark \ref{rem:miranda}
we get that $E$ is a  minimizer of $P_\Phi$
in $\opensetRnplusone.$
\end{remark}

\begin{example}\label{ex:minEx}
For any $l,\gamma\in\R$ we define  
the cones\footnote{A set $E\subseteq \R^m$ 
is a {\it cone } if there exists $x_0\in\p E$ such that for any 
$x\in E$ and $\lambda>0$ it holds $x_0+\lambda (x-x_0)\in E.$} in $\R^{n+1}$
\begin{equation}\label{mincones}
C_1^{(n)}(l,\gamma) := (-\infty, l)\times \R^{n-1}\times(\gamma,+\infty),\quad
C_2^{(n)}(l,\gamma) := (l,+\infty)\times\R^{n-1}\times(-\infty,\gamma). 
\end{equation}
From Example \ref{ex:half-spaces}
and Remark \ref{rem:minimal_sections} it follows that
the following sets are minimizers of $P_\Phi$ in $\R^{n+1}$
provided that $\Phi$ satisfies \eqref{eq:phi_dual_conical}:
\begin{itemize}
\item[a)] $C_1^{(n)}(l_1,\gamma_1)\cup C_2^{(n)}(l_2,\gamma_2)\subset\R^{n+1},$
where $l_1\le l_2,$ $\gamma_1\ge\gamma_2$ (see Figure \ref{doublecone}).
\begin{figure}[ht]
\includegraphics[scale=0.4]{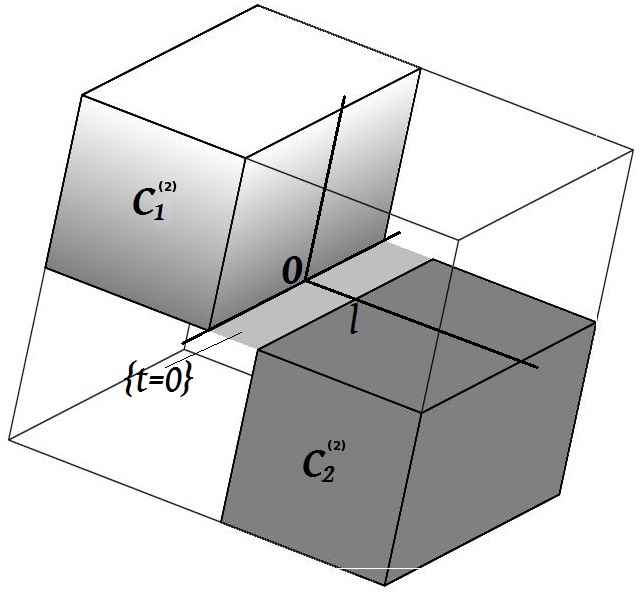}\qquad 
\includegraphics[scale=0.4]{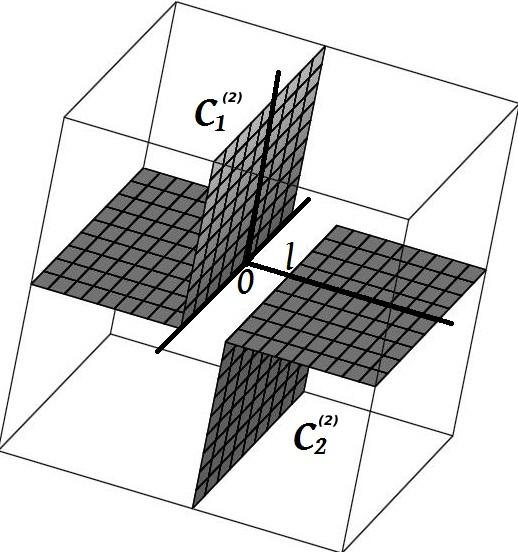}
\caption{$C_1^{(2)}(0,0)\cup C_2^{(2)}(l,0)$ with $l>0$ in Example \ref{ex:minEx}(a) 
and its boundary. The right picture is a slight rotation of left picture.}
\label{doublecone}
\end{figure}
\item[b)] The union of  $C_1^{(n)}(l_1,\gamma)$ and the rotation of 
$C_2^{(n)}(l_2,\gamma)$ around the 
vertical  axis $x_{n+1}$ of $\alpha$ radiants (see Figure \ref{U_Cones}).
 \begin{figure}[ht]
 \includegraphics[scale=0.35]{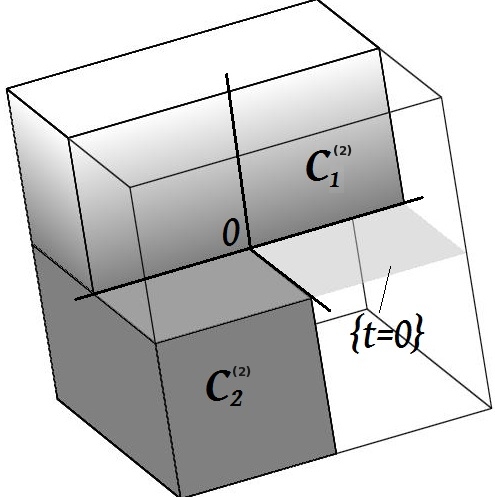}
\caption{Union $C$ of $C_1^{(2)}(0,0)$ and the 
$(-\pi/2)$-rotation of $C_2^{(2)}(0,0)$ in Example \ref{ex:minEx}(b).
Notice that $C_t$ for $t=0$ is not a minimizer of the 
Euclidean perimeter in $\R^2$; 
however, this does not affect the minimality of $C.$}\label{U_Cones}
 \end{figure}

 \end{itemize}
\end{example}

In general, a minimizer of $P_\Phi$ in $\opensetRnplusone$ for a 
cylindrical $\Phi,$ 
need {\it not}  satisfy the minimality property of  horizontal
sections  in Remark \ref{rem:minimal_sections}. 

\begin{example}[\textbf{Strips}]\label{ex:2dim_strip} Let $n=2,$ and let 
$\opensetRn=\R\times (0,\gamma)\subset\R^2$ with $\gamma>0.$ Take 
$\phisec^o(\xi_1^*,\xi_2^*) = |\xi_1^*|+|\xi_2^*|,$ so that 
$$ 
P_\phisec(\Fn,\An) = \int_{\An} |D_{x_1}\chi_{\Fn}|+ \int_{\An} 
|D_{x_2}\chi_{\Fn}|, \qquad \Fn\in BV_\loc(\opensetRn), 
\ \An\in\cK(\opensetRn).
$$

We prove that if $l>\gamma>0$ then the rectangle $\En = (0,l) \times (0,\gamma)$
is a minimizer of $P_\phisec$ in the strip $\opensetRn.$
Let  $\Fn\in BV_\loc(\opensetRn)$ 
be such that $\En\Delta \Fn\strictlyincluded \An \strictlyincluded \opensetRn.$
Let $L_{x_1}$ stand for the vertical line passing through $(x_1,0).$
If $\cH^1(\Fn\cap L_{x_1})=0$ or
$\cH^1(\Fn\cap L_{x_1})=\gamma$ for some
$0<x_1<l,$ then
$$P_\phisec(\Fn,\An) = P_\phisec(\Fn, \An\cap[(-\infty,x_1)\times(0,\gamma)])+
P_\phisec(\Fn, \An\cap[(x_1,+\infty)\times(0,\gamma)]).$$
Hence 
$$P_\phisec(\Fn, \An\cap[(-\infty,x_1)\times(0,\gamma)])\ge P_\phisec(\En, \An\cap[(-\infty,x_1)\times(0,\gamma)]),
$$
$$P_\phisec(\Fn, \An\cap[(x_1,+\infty)\times(0,\gamma)])\ge P_\phisec(\En, \An\cap[(x_1,+\infty)\times(0,\gamma)]),$$
thus $P_\phisec(\Fn,\An)\ge P_\phisec(\En,\An).$
Now assume that $0<\cH^1(\Fn\cap L_{x_1})<\gamma$ for all $x_1\in (0,l).$
In this case $\int_{\An}|D_{x_2}\chi_{\Fn}|\ge 2l.$ Indeed,
since $\En\Delta \Fn\strictlyincluded \An$,
each vertical line $L_{x_1}, $ $x_1\in(0,l)$ should cross $\p^*\Fn$ at 
least twice.
For a similar reason,  taking into account the  term
$\int_{\An}|D_{x_1}\chi_{\Fn}|$  we may assume that $\En\cap \p^*\Fn$ lies on 
two horizontal parallel lines at distance $\epsilon\in (0,\gamma).$
Then by definition of $\phisec$-perimeter 
$$P_\phisec(\Fn,\An)-P_\phisec(\En,\An) \ge 2l - 2\epsilon\ge 2(l-\gamma)>0.$$
This implies that $\En$ is a  minimizer of $P_\phisec$ in $\opensetRn.$
Notice that every horizontal section of $\En$ is $(0,l),$
which is not a minimizer of the perimeter in $\R.$

Now, let $\Phi^o(\hsecdual,\xi_3^*) = \phisec^o(\hsecdual)+|\xi_3^*|.$
By Proposition
\ref{prop:cylindrical_local_minimizers} (b) below, $\En\times \R$ is a 
minimizer of $P_\Phi$ in $\opensetRn\times\R.$ Since $\Phi$ is symmetric with respect to
relabelling the coordinate axis, the set $E = (0,l)\times \R\times(0,\gamma)$
is also a  minimizer of $P_\Phi$ in $\R\times\R\times(0,\gamma).$ Notice that
every horizontal section of $E$ is a translation of the strip $(0,l)\times\R,$
which is not a minimizer of $P_\phisec$ in $\R^2$ according to Example \ref{ex:parallel_planes}.
\end{example}

\begin{example}\label{ex:Union of cones}
Let $\Phi^o(\xi_1^*,\xi_2^*)=|\xi_1^*|+|\xi_2^*|.$
Given $l,\gamma\in\R$ suppose one of the following:
\begin{itemize}
\item[a)] $l=0;$
\item[b)] $l\ge 0\ge \gamma;$
\item[c)] $l\ge\gamma>0.$ 
\end{itemize}
Then  the set $E=C_1^{(1)}(0,0)\cup C_2^{(1)}(l,\gamma)$
is a  minimizer of $P_\Phi$
in $\R^2$ even though in case (c) for any
$t\in (0,\gamma),$ the horizontal section $E_t$ 
is  not a minimizer of the
perimeter in $\R$ (see \eqref{eq:1_dimensional_minimizers}).

Indeed, if $l\ge 0\ge \gamma$ then $E$ satisfies the property in 
Remark \ref{rem:minimal_sections}. If $l=0$ and $\gamma>0,$ 
then $\R^2\setminus E$ is union of two disjoint cones satisfying
property stated in Remark
\ref{rem:minimal_sections}. Thus, in both cases
$E$ is a minimizer of $P_\Phi$ in $\R^2.$

Assume (c). By Remark \ref{rem:minimal_sections}  both
$C_1=(-\infty,0)\times(0,+\infty)$
and
$C_2= (-\infty,\gamma)\times(l,+\infty)$
are  minimizers of $P_\Phi$ in $\R^2$ (for brevity we do not write the dependence on $l$ and $\gamma$).
Consider arbitrary $F\in BV_\loc(\R^2)$
with $E\Delta F\strictlyincluded (-M,M)^2,$ for some $M>0.$
\begin{figure}[ht]
 \includegraphics[scale=0.4]{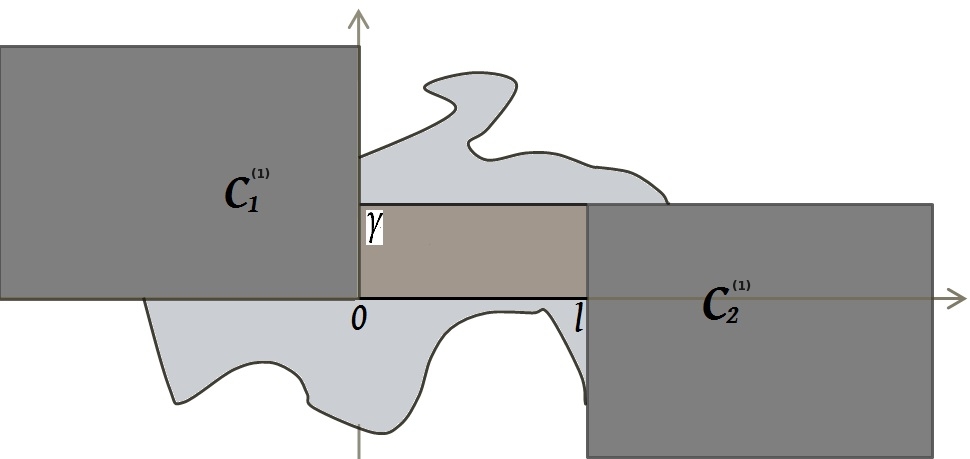}
 \caption{In case $0< \gamma\leq l,$ among all sets connecting two components
 of $E$ the strip parallel to $\xi_1$-axis has the ``smallest'' $\Phi$-perimeter.}\label{bothcomp}
\end{figure}

If $F$ perturbs the components $C_1,C_2$ of $E$ 
separately, {\it i.e.} $F=F_1\cup F_2$
and there exist disjoint open sets $A_1,A_2\subset\R^2$ such that 
$C_i\Delta F_i \subset A_i,$ $i=1,2,$ 
then by minimality of $C_1,$ $C_2$ we have
\begin{align*}
 P_\Phi(F,A_1\cup A_2)  =&P_\Phi(F_1,A_1)+ P_\Phi(F_2,A_2)
\ge  P_\Phi(C_1,A_1)+ P_\Phi(C_2,A_2) 
\\
=&  P_\Phi(E,A_1\cup A_2).
\end{align*}
On the other hand, it is not difficult to see that among all perturbations of $E$
involving both componenents, the best one is obtained by inserting
an  horizontal strip 
as in Figure \ref{bothcomp}. However, because of the assumption 
$0< \gamma\le l,$ this perturbation has larger $\Phi$-perimeter than $E.$
Consequently, $E$ is a  minimizer  of $P_\Phi$ in $\R^2.$
\end{example}

\section{Cylindrical minimizers}\label{sec:cylindrical_minimizers}

Let $\Phi$ be a norm on $\R^{n+1}$ and 
$\opensetRnplusone = \opensetRn \times \R.$

\begin{definition}[\textbf{Cylindrical minimizers}]\label{def:cylindric_min}
We say that 
 a minimizer $E\subseteq\opensetRnplusone$ of $P_\Phi$ in $\Omega$
is {\it cylindrical}   over $\widehat E$ if $E = \En \times
\R,$ where $\En \subseteq \opensetRn.$ 
 
\end{definition}

The aim of this section is to characterize cylindrical
minimizers of $P_\Phi.$ The idea here
is that the (Euclidean) normal to the boundary 
of a cylindrical minimizer is horizontal,
and therefore what matters, in the computation of the anisotropic perimeter, 
is only the horizontal section of the anisotropy.
For this reason it is natural to introduce 
the following property, which informally  requires the upper (and the lower)
part of the boundary of the Wulff shape 
to be a generalized graph (hence possibly with vertical parts) over 
its projection on the horizontal hyperplane $\Rn \times
\{0\}.$

\begin{definition}[\textbf{Unit ball as a generalized graph in the vertical direction}]
We say  that the boundary of the 
unit ball $B_\Phi$ of the norm $\Phi: \R^{n+1} \to [0,+\infty)$
is a {\it generalized graph in the vertical direction}
if
\begin{equation}\label{compRadonMeas}
\Phi(\hsectangent,\xi_{n+1})\ge \Phi(\hsectangent,0),\qquad
(\hsectangent,\xi_{n+1})\in\R^{n+1}.
\end{equation}
\end{definition}

In Lemma \ref{lem:AdjointPropertyG} we show that $\partial B_\Phi$
is a generalized graph in the vertical direction
if and only if so is $\partial B_{\Phi^o}.$

\begin{example}\label{ex:generalized_graph}
\begin{itemize}
 \item[(a)] If
$\Phi(\hsectangent,-\xi_{n+1})=\Phi(\hsectangent,\xi_{n+1})$
for all $(\hsectangent, \xi_{n+1}) \in \R^{n+1},$ then $\partial B_\Phi$  
is a generalized graph in the vertical direction.
Indeed, from convexity
\begin{equation*}
\Phi(\hsectangent,0)\le \Phi\left(\frac{\hsectangent}{2},\frac{\xi_{n+1}}{2}\right)+
\Phi\left(\frac{\hsectangent}{2},-\frac{\xi_{n+1}}{2}\right)=\Phi(\hsectangent,\xi_{n+1}),
\qquad
(\hsectangent, \xi_{n+1})\in\R^{n+1}.
\end{equation*}
\item[(b)] There exists $\partial B_\Phi$ which is a  generalized graph in the vertical direction,
but $\Phi$ does not satisfy
$\Phi(\hsectangent,-\xi_{n+1})=\Phi(\hsectangent,\xi_{n+1}).$
Fix some  $\epsilon\in(1/\sqrt{2},\sqrt{2}]$ and
consider the (symmetric convex)
plane hexagon $K_\epsilon$ with vertices
at $(1,0),$ $(\epsilon,-\epsilon),$ $(0,-1),$
$(-1,0),$ $(-\epsilon,\epsilon),$ $(0,1).$ Let
$\Phi_\epsilon:\R^2\to [0,+\infty)$ be the Minkowski functional of $K_\epsilon.$ Then
$\Phi_\epsilon$  does not satisfy
$\Phi_\epsilon(\xi_1,-\xi_2)=\Phi_\epsilon(\xi_1,\xi_2).$ But
$\partial B_{\Phi_\epsilon}$ is a generalized graph in the vertical direction. Indeed,
consider the  straight line passing through $(1,0)$ and parallel to $\xi_2$ - axis.
This line does not cross the interior of $K_\epsilon.$
Thus $\Phi_\epsilon(1,\xi_2)\ge 1=\Phi_\epsilon(1,0).$
 If $\xi_1\ne 0,$ then
$$
\Phi_\epsilon(\xi_1,\xi_2) = |\xi_1| \Phi_\epsilon(1,\xi_2/\xi_1)
\ge |\xi_1| \Phi_\epsilon(1,0)=\Phi_\epsilon(\xi_1,0).$$
If $\xi_1=0,$ the inequality $\Phi_\epsilon(\xi_1,\xi_2)\ge
\Phi_\epsilon(\xi_1,0)$ is obvious.

\item[c)] The norm $\Phi:\R^2\to [0,+\infty),$
 $\Phi(\xi_1,\xi_2)=\sqrt{\xi_1^2+\xi_1\xi_2+\xi_2^2}$
has a unit ball the boundary of which is not 
 a generalized graph in the vertical direction, since
$\Phi(2,0) = 2>\sqrt{3} = \Phi(2,-1).$
\end{itemize}
\end{example}

\begin{proposition}[\textbf{Cylindrical minimizers}]\label{prop:cylindrical_local_minimizers}
Let $\Phi:\R^{n+1}\to [0,+\infty)$ be a norm.
Let $\En
\in BV_{\rm loc}(\opensetRn).$
The following
assertions hold:
\begin{itemize}
\item[(a)] if 
$\En\times \R$	
is a  minimizer of $P_\Phi$ in $\opensetRn \times \R,$ then
$\En$ is a   minimizer of $P_\phi$
in $\opensetRn,$ where
$$\phi:= \left(\Phi_{\vert_{\{\xi_{n+1}^*=0\}}}^o \right)^o.$$
\item[(b)]
if $\partial B_{\Phi}$ is a generalized graph in the vertical direction and
$\En$ is a minimizer of $P_\phisec$ in $\opensetRn,$ where
 $\phisec:=\Phi_{\vert_{\{\xi_{n+1}=0\}}}^{},$ 
then
$\En\times \R$ is a  minimizer of $P_\Phi$
in $\opensetRn\times \R.$
\end{itemize}
\end{proposition}
\begin{remark}
In general $\phi\ne \phisec$ (see Remark \ref{rem:StrInq} and Lemma \ref{lem:CommutRestAdj}).
\end{remark}

\begin{proof}
(a) Take $\bopensetRn\in\cK(\opensetRn),$
 $\Fn \in BV_{\rm loc}(\opensetRn)$ with
$\En\Delta \Fn\strictlyincluded \bopensetRn.$
For any  $m>0$ set $I_m := (-m,m),$ and
define
$$
F_m := [E\setminus (\R^n
\times I_m)] \cup [\Fn\times I_m].$$
Then
 $E\Delta F_m\strictlyincluded \bopensetRn \times I_{m+1} \subset
\subset \opensetRn\times \R$
 and, by minimality,
\begin{equation}\label{eq:minimality}
P_\Phi(E,\bopensetRn \times I_{m+1}) \le
P_\Phi(F_m, \bopensetRn \times I_{m+1}).
\end{equation}
Writing $\nu_E^{} = (\widehat{\nu_E}, (\nu_E)_t),$
we have   $\nu_E^{}=(\nu_{\En},0)$ $\mathcal H^n$-almost
everywhere on $\p^*E.$
Hence
\begin{equation}\label{eq:P_Phi_E_K}
 \begin{aligned}
  P_\Phi\left(E,\bopensetRn\times I_{m+1}\right)
= &\int_{[\bopensetRn\times I_{m+1}]
\cap \partial^* E}
\Phi^o\left(\widehat{\nu_E}, (\nu_E)_t\right)\,d\cH^{n}
\\
  =& \int_{[\bopensetRn\times I_{m+1}] \cap \partial^* E}
\phi^o( \nu_{\En} )\,d\cH^{n}
= 2(m+1)\int_{\bopensetRn \cap \p^*\En} \phi^o(\nu_{\En})\,d\cH^{n-1}
\\
=& 2(m+1)P_\phi(\En,\bopensetRn).
 \end{aligned}
\end{equation}
Similarly, $\nu_{F_m}=(\nu_{\Fn},0)$ on  $(\p^* \Fn)\times I_m,$
$\nu_{F_m} = (0,\pm1)$ on
$(\En\Delta \Fn) \times \{\pm m\}$ and $\nu_{F_m}=(\nu_{\En},0)$ on
$(\p^* \Fn)\times (I_{m+1}\setminus\overline{I_m}).$
As a consequence,
\begin{equation}\label{eq:P_Phi_F_K}
\begin{aligned}
  P_\Phi(F_m, \bopensetRn\times I_{m+1})
=& \int_{[ \bopensetRn\times I_{m} ] \cap \partial^* F_m}
\Phi^o(\nu_{F_m})\,d\cH^{n} + \int_{[ \bopensetRn\times \{\pm m\} ] \cap \partial^* F_m}
\Phi^o(\nu_{F_m})\,d\cH^{n}
\\
&+\int_{[ \bopensetRn\times (I_{m+1}\setminus \overline{I_m}) ] \cap \partial^* F_m}
\Phi^o(\nu_{F_m})\,d\cH^{n}
\\
=&
2m P_{\phi}(\Fn, \bopensetRn) + 2 \Phi^o(0,1)\cH^{n}(\En\Delta \Fn)+2P_\phi(\En,\An).
 \end{aligned}
\end{equation}
From
\eqref{eq:P_Phi_E_K},
\eqref{eq:P_Phi_F_K} and
\eqref{eq:minimality}, it follows
$$P_\phi(\En,\bopensetRn)\le P_\phi(\Fn,\bopensetRn) +
\dfrac{\Phi^o(0,1)}{m}\,\cH^{n}(\Fn\Delta \En).
$$
Letting $m\to+\infty$  we get $P_\phi(\En,\bopensetRn)\le
 P_\phi(\Fn,\bopensetRn),$ and
assertion (a)
follows.

\smallskip

(b) 
By Lemma \ref{lem:CommutRestAdj}, $\phisec^o=\Phi_{\vert_{\{\xi_{n+1}^*=0\}}}^o.$
Take $F \in BV_{\rm loc}(\opensetRn\times \R),$
and let $\bopensetRn \in\cK(\opensetRn)$ and $M>0$
be such that
 $E\Delta F \strictlyincluded \bopensetRn \times I_M,$ where
$I_M := (-M,M).$ 
Then $\En\Delta F_t\strictlyincluded \An$
for all $t\in(-M,M)$ and
since $\En$ is a minimizer of $P_\phisec$
in $\opensetRn,$ using  \eqref{compRadonMeas} and \eqref{eq:hor_per}
we get
\begin{align*}
P_\Phi(F,\bopensetRn \times I_M)
=& \int_{(\bopensetRn \times I_M)\cap\p^*F} \Phi^o( \widehat{\nu_F}, (\nu_F)_t )\,d\cH^{n} \ge
\int_{(\bopensetRn \times I_M)\cap\p^*F} \Phi^o( \widehat{\nu_F}, 0 )\,d\cH^{n}\\
=&  \int_{-M}^M P_\phisec(F_t, \bopensetRn)\,dt\ge
 \int_{-M}^M P_\phisec(\En, \bopensetRn)\,dt 
 =  P_\Phi(E,\An\times I_M),
 \end{align*}
and assertion (b) follows.
\end{proof}

\begin{example}[\textbf{Characterization of cylindrical minimizers 
for a cubic anisotropy}]\label{ex:char.cylind.min}\rm
Proposition \ref{prop:cylindrical_local_minimizers} allows us to
classify the cylindrical minimizers
of $P_\Phi$  for suitable  choices of the dimension and of the anisotropy.
Take $n=2,$ $\opensetRn=  \R^2,$
and let 
$$
\unitballPhi = [-1,1]^3;
$$
 in particular,
$\partial B_{\Phi}$ is a generalized graph in the vertical direction and
$B_\phisec$ is the 
square $[-1,1]^2$ in the (horizontal) plane. The  minimizers of $P_\phisec$
are classified
as follows \cite[Theorems 3.8 (ii) and  3.11 (2)]{NoPa:05}: the infinite
cross $\widehat C = \{|x_1| >|x_2|\}$ and its complement,
the subgraphs and epigraphs $\widehat S$
of monotone functions of one variable, and suitable unions $\widehat U$ of two
connected components, each of which is the subgraph of a monotone
function of one variable.
Then  Proposition
\ref{prop:cylindrical_local_minimizers} (b) implies that
$$
\widehat C \times \R, \quad (\R^2\setminus \widehat C) \times \R, 
\quad \widehat S \times \R, \quad \widehat U \times \R
$$
are the {\it only} cylindrical minimizers of $P_\Phi$ in $\R^3.$
The same result holds if $B_\phisec$ is a parallelogram centered at the origin,
and $\partial B_\Phi$ is any generalized graph in the vertical direction
such that $B_\phisec = B_\Phi\cap \{\xi_3=0\}.$
\end{example}

\section{Cartesian minimizers for partially monotone norms}\label{sec:cartesian_minimizers}
Let $\Phi:\R^{n+1} \to [0,+\infty)$ be a norm. 

\begin{definition}[\bf Cartesian minimizers]\label{def:cartesian_min}
We call a minimizer $E\subseteq \Omega = \opensetRn\times\R$ a {\it cartesian minimizer} of 
$P_\Phi$ in $\Omega = \opensetRn \times\R$ if $E = \sg(u)$
for some function $u : \opensetRn \to \R.$ 
\end{definition}

Let $v\in BV_{\rm loc}(\opensetRn)$; in what follows the symbol $\int_{\An}\Phi^o(-Dv,1)$
means
\begin{equation*}
\begin{aligned}
& \int_{\bopensetRn} \Phi^o( - Dv,1) = \sup\Big\{ \int_{\An} \Big(v
\sum\limits_{j=1}^n\frac{\partial \eta_j}{\partial x_j}
 +\eta_{n+1} \Big)\,dx:
\\
&\qquad \qquad \qquad \qquad  \qquad  \eta=(\eta_1,\ldots,\eta_{n+1})\in C_c^1(\An,B_\Phi)
\Big\}.
\end{aligned}
\end{equation*}
If $v\in W_\loc^{1,1}(\opensetRn)$ we have
\cite[Theorem 2.91]{AFP:00} 
\begin{equation*}
\begin{aligned}
 P_\Phi(\sg(v),\bopensetRn\times\R)
=&
\int_{(\bopensetRn\times\R)\cap \partial \sg(v) }
\Phi^o(\nu_{\sg(v)})\,d\cH^n
\\
 =&\int_{(\bopensetRn\times\R) \cap \partial \sg(v)}
\Phi^o(-\nabla v,1)\,\frac{d\cH^n}{\sqrt{1+|\nabla v|^2}}
 =\int_{\bopensetRn} \Phi^o(-\nabla v,1)\,dx.
\end{aligned}
\end{equation*}
Using the techniques in
\cite{Da:80},
the previous equality extends to
any $v\in BV_{\rm loc}(\opensetRn),$
\begin{equation}\label{minP}
 P_\Phi(\sg(v),\An\times\R)=
\int_{(\bopensetRn\times\R)\cap \partial^* \sg(v) }
\Phi^o(\nu_{\sg(v)})\,d\cH^n=
\int_{\bopensetRn} \Phi^o(-Dv,1).
\end{equation}
Accordingly, we define the functional
$\cG:BV_{\rm loc}(\opensetRn) \times \cK(\opensetRn)\to [0,+\infty)$
as follows:
$$
 \cG(v, \bopensetRn):=\int_{\bopensetRn} \Phi^o(-Dv,1),\qquad
v\in BV_\loc(\opensetRn),\,
\bopensetRn\in\cK(\opensetRn).
$$
\begin{definition}\label{def:local_minimizer_of_G}
We say that $u\in BV_{\rm loc}(\opensetRn)$ is a {\it minimizer} of
$\cG$ by {\it compact perturbations} in $\opensetRn$
(briefly, a {\it minimizer} of $\cG$ in $\opensetRn$),
and we write 
$$
u \in \cM_{\Phi^o}(\opensetRn),
$$
if for any $\bopensetRn\in\cK(\opensetRn)$ and 
$v\in BV_\loc(\opensetRn)$ with $\supp (u-v) \strictlyincluded \bopensetRn$ one has
$$\cG(u,\bopensetRn)\le \cG(v, \bopensetRn).
$$
\end{definition}

Note that
 $\cM_{\Phi^o}(\opensetRn)\ne \emptyset$ since linear functions
on $\opensetRn$ belong\footnote{
If $u$ is linear,
then  $\sg(u)$ is the intersection
of a half-space with $\opensetRn\times\R,$ hence
$\sg(u)$ is a
minimizer of $P_\Phi$  in $\opensetRn\times\R$
(Example \ref{ex:half-spaces}) and
$u\in\cM_{\Phi^o}(\opensetRn)$ (see Theorem \ref{teo:Subgraphs}(a) below).}
 to $\cM_{\Phi^o}(\opensetRn).$
Observe also that if
 $u\in\cM_{\Phi^o}(\opensetRn)$ then $u+c\in\cM_{\Phi^o}(\opensetRn)$
for any $c\in\R.$

We shall need the following standard result.
\begin{theorem}[\textbf{Compactness}]\label{teo:compactness}
Let $\Phi:\R^{n+1}\to [0,+\infty)$ be a norm.
If $u_k\in\cM_{\Phi^o}(\opensetRn),$ $u\in L_{\loc}^1(\opensetRn)$ and $u_k\to u$ in $L_{\loc}^1(\opensetRn)$
as $k \to +\infty,$
then $u\in\cM_{\Phi^o}(\opensetRn).$
\end{theorem}

\begin{proof}
The proof is the same as in \cite[Theorem 3.4]{NoPa:02} making use of lower
semicontinuity of
$P_\Phi,$ \eqref{comparison} and the inequality
$\Phi^o(-Dw,1)\le \Phi^o(Dw,0)+\Phi^o(0,1).$ 
\end{proof}

The aim of this section is to show the relations between 
minimizers and cartesian minimizers, under a special assumption
on the norm. 

\begin{definition}[\textbf{Partially monotone norm}]\label{def:partial_monotone_norm}
The norm $\Phi: \R^{n+1} \to [0,+\infty)$  is called {\it partially monotone} if
given $\xi = (\hsectangent,\xi_{n+1})\in
\R^{n+1}$ and $ \eta =(\widehat{\eta},\eta_{n+1})\in\R^{n+1}$ we have
\begin{equation}\label{comp_xi_eta}
\Phi(\widehat{\xi},0)\le \Phi(\widehat{\eta},0),  \ \
\Phi(0,\xi_{n+1})\le \Phi(0,\eta_{n+1}) \quad \Longrightarrow
\quad \Phi(\xi)\le \Phi(\eta). 
\end{equation}
\end{definition}

Partially monotone norms are characterized in Section \ref{subsec:partially_monotone_norms}.

\begin{example}\label{ex:Part_mon}
The following norms on $\R^{n+1}$ are partially monotone:
$\Phi(\hsectangent,\xi_{n+1}) = \max\{\phisec(\hsectangent),|\xi_{n+1}|\}$;
$\Phi(\hsectangent,\xi_{n+1}) 
= ([\phisec(\hsectangent)]^p+|\xi_{n+1}|^p)^{1/p},$ 
 where $\phisec:\Rn \to \oic$
 is a norm and  $p\in[1,+\infty).$
\end{example}

\begin{theorem}[\textbf{Minimizers and cartesian minimizers}]\label{teo:Subgraphs}
Let $\Phi:\R^{n+1} \to [0,+\infty)$
 be a norm,
and $u\in BV_{\rm loc}(\opensetRn).$
The following assertions hold:
\begin{itemize}
 \item[(a)]  if $\sg(u)$ is a  minimizer of $P_\Phi$ 
in $\opensetRn\times\R,$ then
 $u$ is a minimizer of $\cG$ in $\opensetRn$;
 \item[(b)] if $\Phi$ is partially monotone
and  $u$ is a minimizer of $\cG$  in $\opensetRn,$
then $\sg(u)$ is a  minimizer of $P_\Phi$
in $\Omega = \opensetRn\times\R.$
\end{itemize}
\end{theorem}
\begin{proof}
(a)
Let $\test\in
 C_c^1(\opensetRn)$
be such that
 $\supp(\test)\strictlyincluded \bopensetRn$ for
some $\bopensetRn\in\cK(\opensetRn).$ Then there exists $H>0$ such that
$\sg(u) \Delta \sg(u+\test) \strictlyincluded \bopensetRn\times(-H,H).$
If $\sg(u)$ is a  minimizer of $P_\Phi,$ then
$P_\Phi(\sg(u), \bopensetRn\times\R)\le 
P_\Phi(\sg(u+\test),\bopensetRn\times\R)$ and
so, by virtue of \eqref{minP},
\begin{equation}\label{eq1234}
\cG(u,\An)\le \cG(u+\test,\An). 
\end{equation}
For general $\test \in BV_\loc(\opensetRn)$ inequality \eqref{eq1234} can 
be proven by approximation.
\smallskip

(b) Let $u$ be a minimizer of
$\cG$ and  $F \in BV_{{\loc}}(\Omega)$ 
be such that $\sg(u)\Delta F\strictlyincluded A=\An\times(-M,M)$ with
$\An\in\cK(\opensetRn)$ and $M>0.$ Then \eqref{impassump} yields that
$$P_\Phi(F, \Bn\times \R) <+\infty \qquad \forall \Bn\in\cK(\opensetRn).
$$

We shall closely follow \cite{Gi:84, MM1964}, where
the argument is done in the Euclidean setting. 
For simplicity let $\phisec_1^o(\cdot)=\Phi^o(\cdot,0)$
and $\phisec_2^o(\cdot)=\Phi^o(0,\cdot).$ 
We claim that there exists $v\in BV_{\loc}(\opensetRn)$ with
$\supp (u-v)\strictlyincluded \An$ such that for any $\Bn\in \cK(\opensetRn)$
\begin{equation}\label{partper2}
\begin{aligned}
\int_{\Bn\times\R} \phisec_1^o(D_x\chi_{\sg(v)}) \le  \int_{\Bn\times\R} \phisec_1^o(D_x\chi_F),\\
\int_{\Bn\times\R} \phisec_2^o(D_t\chi_{\sg(v)})) \le  \int_{\Bn\times\R} \phisec_2^o(D_t\chi_F).
 \end{aligned}
\end{equation}
Supposing that the claim is true,
from \eqref{partper2} and from Lemma \ref{lem:PartmonNorm}
we deduce
$$
P_\Phi(\sg(v),\An\times\R) = \int_{\An\times\R} \Phi^o(D\chi_{\sg(v)}) 
\le \int_{\An\times\R} \Phi^o(D\chi_F)=P_\Phi(F,\An\times\R).$$
 Then by the minimality of $u$ and \eqref{minP}
 we get
\begin{align*}
 P_\Phi(\sg(u),\bopensetRn\times\R)=&\int_{\bopensetRn} \Phi^o(-Du,1)
\le \int_{\bopensetRn} \Phi^o(-Dv,1)
\\
 =&P_\Phi(\sg(v),\bopensetRn\times\R)\le P_\Phi(F,\bopensetRn\times\R).
 \end{align*}

Let us prove our claim. Since $\sg(u) \Delta F \subset \subset A,$ we have
\begin{equation}\label{subgraph}
\lim\limits_{t\to +\infty} \chi_F(x,t) = 0,\,\,\,
\lim\limits_{t\to -\infty} \chi_F(x,t) = 1
\quad \text{for a.e. $x\in\opensetRn.$}
\end{equation}
Then,
 by \cite[Lemma 14.7 and Theorem 14.8]{Gi:84} (see also \cite[Theorem 2.3]{MM1964})
the function
$$v_h(x) := \int_{-h}^h \chi_F(x,t) dt -h, \qquad x \in \opensetRn,
$$
belongs to $L_\loc^1(\opensetRn)$ and the sequence $\{v_h\}$ converges
 pointwise
to  $v\in BV_\loc(\opensetRn)$ as $h \to +\infty.$
To show  that $u-v$ is compactly
supported in $\An$ it is enough to
take  $\An'\in\cK(\opensetRn)$
such that $\An'\strictlyincluded\An$ and 
$\sg(u) \Delta F \subset \subset \An'\times(-M,M),$ and to observe that
since $\sg(u)\cap \big((\opensetRn\setminus \An')\times\R\big) =
F\cap \big((\opensetRn\setminus \An')\times\R\big),$ if $x\in \opensetRn\setminus \An',$
for $h$ sufficiently large  we have
\begin{align*}
v_h(x)=& \int_{-h}^{u(x)} \chi_F(x,t)dt -h= u(x).
\end{align*}

Now, define  $\eta_h:\R\to [0,+\infty)$ as $\eta_h := 1$ on $[-h,h],$
$\eta_h :=0$ on $(-\infty, -h-1]\cup [h+1,+\infty),$ and
\begin{equation*}
\eta_h(t) :=
\begin{cases}
h+1-t & {\rm if}~ h\le t\le h+1,\\
h+1+t & {\rm if}~ -h-1\le t\le -h.
\end{cases}
\end{equation*}
Being
 $1/2=\int_{-h-1}^{-h}[h+1+t]dt,$ we have
\begin{align*}
\Big| \int_\R \eta_h(t)\chi_F(x,t)dt -h - & \frac12 - v(x)\Big| 
 =  \Big|-\int_{-h-1}^{-h}[h+1+t](1-\chi_F(x,t))dt\\
+ & \int_h^{h+1} [h+1-t]\chi_F(x,t)dt+v_h(x) -v(x)\Big|
\\
 \leq & \Big|v_h(x)  -v(x)\Big|
 + \int_{-h-1}^{-h} (1-\chi_F(x,t))dt +
\int_{h}^{h+1} \chi_F(x,t)dt.
\end{align*}
Hence,
from
 \eqref{subgraph},
for almost every $x\in\opensetRn$ we get
\begin{equation}\label{domconvth}
 \lim\limits_{h\to+\infty}
 \Big| \int_\R \eta_h(t)\chi_F(x,t)dt -h - \frac12 - v(x)\Big| =0. 
\end{equation}
Let us fix $\psi\in C_c^1(\opensetRn)$ and $1\le j\le n.$
Then, using 
$\int_{\opensetRn}  D_{x_j}\psi(x)dx=0,$
the dominated convergence theorem (see \cite{Gi:84,MM1964} for more details)
and \eqref{domconvth}
we find
\begin{align*}
\int_{\opensetRn\times\R} \psi(x) D_{x_j}\chi_F(x,t) = &
\lim\limits_{h\to+\infty}
\int_{\opensetRn\times\R} \eta_h(t) \psi(x) D_{x_j}\chi_F(x,t)
\\
=&- \lim\limits_{h\to+\infty} \int_{\opensetRn} D_{x_j} \psi(x)dx \int_\R
\eta_h(t) \chi_F(x,t) dt
\\
=& - \lim\limits_{h\to+\infty}
\int_{\opensetRn} D_{x_j} \psi(x) \left[\int_\R
\eta_h(t) \chi_F(x,t)dt -h-\frac12\right]~dx\\
=& -\int_{\opensetRn} v(x) D_{x_j}\psi(x)dx.
\end{align*}
Hence for any $\An\in\cK(\opensetRn) $ and $\eta\in C_c^1(\An;B_{\phisec_1})$  one has
\begin{align*}
-\int_{\An} v(x) \sum\limits_{j=1}^n D_{x_j}\eta(x)dx= &
\int_{\An\times\R} \eta(x)\cdot D_x\chi_F(x,t)
\le  \int_{\An\times\R} \phisec_1^o(D_x\chi_F(x,t)).
\end{align*}
Since $\eta$ is
arbitrary, the definition of $\int_{\An} \phisec_1^o(D_{x}v)$  implies
\begin{equation}\label{eq:intphisec}
 \int_{\An} \phisec_1^o(D_{x}v) \le \int_{\An\times\R}
\phisec_1^o(D_{x}\chi_F).
\end{equation}
Being $|D_t\chi_F|$ a counting measure, we have \cite{MM1964}
\begin{equation}\label{eq:phi_2}
\int_{\An\times\R} \phisec_2^o(D_t\chi_F) =    \phisec_2^o(1) \int_{\An\times\R} |D_t\chi_F| \ge
  \phisec_2^o(1) |\An|.
 \end{equation}
Moreover, one checks that
 \begin{equation}\label{eq:onechecks}
\int_{\An\times\R} \phisec_1^o(D_{x}\chi_{\sg(v)}) =
 \int_{\An} \phisec_1^o(D_{x}v),\qquad
 \int_{\An\times\R} \phisec_2^o(D_{t}\chi_{\sg(v)}) =
 \phisec_2^o(1)\,|\An|.
\end{equation}
Now our claim \eqref{partper2} follows from \eqref{eq:intphisec}-\eqref{eq:onechecks}.
\end{proof}

\begin{corollary}\label{cor:locally finiteness}
Let $\Phi^o: ({\R^{n+1}})^* \to [0,+\infty)$ 
be a partially monotone norm and $u\in \mathcal M_{\Phi^o}(\opensetRn).$
Then $u \in L^\infty_{\rm loc}(\opensetRn).$
\end{corollary}

\begin{proof}
It follows
repeating essentially
the same arguments in the proof of \cite[Theorem 14.10]{Gi:84},
using Theorem \ref{teo:Subgraphs}(b) and the density estimates
(see for instance \cite[Proposition 1.10]{Me:15} for 
the anisotropic setting).
 \end{proof}

\section{Classification of cartesian minimizers for cylindrical norms}\label{sec:cartesian_minimizers_for_cylindrical_norms}

The aim of this section is to give a rather complete
classification of entire cartesian minimizers, supposing
the norm $\Phi$  cylindrical. As explained in
the introduction, this case covers, in particular,
the study of minimizers of the total variation functional. 
We start with a couple of observations. 

\begin{remark}\label{rem:mult_scalar}
Suppose  that $\Phi: \R^{n+1}\to \oic$ is cylindrical over $\phisec.$
Then 
\begin{equation}\label{eq:mult_scalar}
u\in\cM_{\Phi^o}(\opensetRn) \Longrightarrow \lambda u\in\cM_{\Phi^o}(\opensetRn) \qquad \forall \lambda\in\R,
\end{equation}
since
\begin{equation*}
\cG(v, \bopensetRn)=\int_{\bopensetRn} \phisec^o(Dv)+ |\An|,
\qquad
(v,\bopensetRn)\in BV_{\rm loc}(\opensetRn)\times \cK(\opensetRn).
\end{equation*}
On the other hand,
\eqref{eq:mult_scalar} is expected to hold not for all
non cylindrical norms $\Phi.$
For example, let $\Phi$ be Euclidean, $n\ge 8$ and
$u:\R^n\to\R$ be a smooth nonlinear solution \cite{BDG:69}
of the minimal surface equation
${\rm div} \left( \frac{\nabla u}{\sqrt{1+|\nabla u|^2}} \right) =0.$
Then\footnote{$u$ is a minimizer of
${\mathcal G}_{\Phi^o}$ in  $\Rn,$
since the Euclidean unit normal (pointing upwards) to ${\rm graph}(u),$
constantly extended in the $e_{n+1}$ direction, provides a
calibration for ${\rm graph}(u)$ in the whole of
$\R^{n+1}.$} 
$u\in\cM_{\Phi^o}(\R^n),$ but if $\Delta u \vert
\nabla u\vert^2$ is not identically zero, then 
 $\lambda u\notin\cM_{\Phi^o}(\R^n)$ for any $\lambda\in \R\setminus\{ 0,\pm 1\}.$
Indeed, otherwise  $\lambda u$ solves the minimal surface equation, hence
\begin{align*}
0=&\lambda
{\rm div}\left(\frac{\nabla u}{\sqrt{1+\lambda^2|\nabla u|^2}}\right)
\\
=& \frac{\lambda}{\sqrt{1+\lambda^2|\nabla u|^2}}
\left(\Delta u -\frac{\lambda^2}{{1+\lambda^2|\nabla u|^2}}
 \sum\limits_{i,j=1}^n 
\nabla_i u\cdot \nabla_j u
\nabla_{ij}u\right)
\\
=& \frac{\lambda}{\sqrt{1+\lambda^2|\nabla u|^2}}\left(\Delta u -
\lambda^2\Delta u \frac{1+|\nabla u|^2}{1+\lambda^2|\nabla u|^2}\right)
=\frac{\lambda(1-\lambda^2)\Delta u \,|\nabla u|^2}{(1+\lambda^2|\nabla u|^2)^{3/2} }.
\end{align*}
If $\Delta u|\nabla u|^2$ is not identically zero, we get
$\lambda(1-\lambda^2)=0,$ a contradiction.
\end{remark}

\begin{remark}\label{rem:max_min}
Suppose  that $\Phi: \R^{n+1}\to \oic$ is cylindrical over $\phisec.$
Then
\begin{equation}\label{eq:sections}
 u\in\cM_{\Phi^o}(\opensetRn)
~ \Longrightarrow~ \max\{u,\level\},\,\min\{u,\level\}\in \cM_{\Phi^o}(\opensetRn) \quad \forall\level\in\R.
\end{equation}
Indeed, suppose first  $\level=0.$ If
$\{u\ge 0\}\in BV_\loc(\opensetRn),$ 
then \eqref{eq:sections} can be proven as in
\cite[Lemma 3.5]{NoPa:02}, using
\cite[Theorem 3.84]{AFP:00}.
In the general case,
 by the coarea formula there exists a sequence $\level_j\uparrow 0$ such that
$\{u\ge \level_j\}\in BV_\loc(\opensetRn).$
Clearly $u_j :=u-\level_j\in\cM_{\Phi^o}(\opensetRn),$ hence
$u_j^+\in\cM_{\Phi^o}(\opensetRn).$ Since $u_j^+\to u^+$ in $L_\loc^1(\opensetRn),$
Theorem \ref{teo:compactness}  implies  $u^+\in\cM_{\Phi^o}(\opensetRn).$
The case $\level \neq 0$ is implied by the
previous proof  and the identity
$\max\{u,\level\} = (u-\level)^++\level.$
The relation $\min\{u,\level\}\in\cM_{\Phi^o}(\opensetRn)$ then 
follows from the identity
$\min\{u,\level\} = -\max\{-u,-\lambda\}$ and 
from Remark \ref{rem:mult_scalar}.
\end{remark}

Further properties of cartesian minimizers are
listed in the following proposition, which in particular (when $\phisec$ is Euclidean) asserts
some properties of minimizers of the total variation functional \cite{Cham:09}.

\begin{proposition}[\textbf{Cartesian minimizers for cylindrical norms}]\label{prop:cartesian_min_cyl}
Suppose  that $\Phi: \R^{n+1}\to \oic$ is cylindrical over $\phisec.$
The following assertions hold:
\begin{itemize}
\item[(a)] 
if $u\in\cM_{\Phi^o}(\opensetRn)$
and $\level \in \R$ then
$\chi_{\{u > \level\}}, \chi_{\{u \geq \level\}} 
\in \cM_{\Phi^o}(\opensetRn)$;
\item[(b)] 
if $\En\subset\opensetRn$ and
$\chi_{\En}\in\cM_{\Phi^o}(\opensetRn)$
then $\En$ is a minimizer of $P_\phisec$  in $\opensetRn;$
\item[(c)] 
if $u\in\cM_{\Phi^o}(\opensetRn)$ and $\level\in\R$ then
$\{u>\level\}$ and $\{u\ge \level\}$ are
minimizers  of $P_\phisec$ in $\opensetRn$;
\item[(d)]  if $u\in BV_\loc(\opensetRn)$
and for almost every $\level \in\R$ the sets  $\{u>\level\}$ 
(resp. $\{u\ge \level\}$) are minimizers of 
$P_\phisec$ in $\opensetRn,$ then 
$u\in \cM_{\Phi^o}(\opensetRn)$;
\item[(e)] 
if $u\in\cM_{\Phi^o}(\opensetRn)$ and
$f : \R \to \R$ is monotone
then $f\circ u \in\cM_{\Phi^o}(\opensetRn)$;
\item[(f)]
let $\direction\in\R^n,$ $f:\R\to\R$ be
a  monotone function, and define
$u(x) := f(x\cdot \direction)$ for any $x\in \opensetRn.$ Then $u \in\cM_{\Phi^o}(\opensetRn).$
\end{itemize}
\end{proposition}

Clearly, assertion (e) generalizes \eqref{eq:mult_scalar} and \eqref{eq:sections}.
We also anticipate here that the converse of statement (f) is considered in Theorem \ref{teo:Euc} below.

\begin{proof}
The proof of (a) 
is the same as in \cite[Theorem 1]{BDG:69}
and (b) is immediate.
(c) follows from (a) and (b), while (d) follows from the coarea
formula
$$
\int_{\An}\phisec^o(Dv) = \int_\R P_\phisec(\{v>\lambda\},\An)\,d\lambda,
\qquad v \in BV(\An).
$$
Let us prove (e). Without loss of generality assume that $f$ is nondecreasing.
Suppose first that $f$ is Lipschitz and strictly increasing.
Set $v=f\circ u,$ and let
$\level\in\R.$
Since $u\in\cM_{\Phi^o}(\opensetRn),$
by (c) it follows that  $\{u\ge f^{-1}(\level)\}$ and $\{u>f^{-1}(\level)\}$
are minimizers
of $P_\phisec$ in $\opensetRn,$ hence 
$\{v>\level\}$ are minimizers of $P_\phisec$ in $\opensetRn.$ Then (d) implies
 $v\in \cM_{\Phi^o}(\opensetRn).$ In the general case,
it is sufficient to approximate $f$ with a sequence of 
strictly increasing Lipschitz functions, and use Theorem \ref{teo:compactness}.
(f) follows from (e), since the
linear function $u_0(x)=x\cdot\direction,$ $x\in\opensetRn,$ is a minimizer
of $\cG$ in $\opensetRn.$
\end{proof}

Now, we show that Proposition \ref{prop:cartesian_min_cyl}(f) implies the minimality of certain cones; 
the same conclusion could be obtained by applying Remark \ref{rem:minimal_sections}.

\begin{proposition}[\textbf{Cones minimizing the anisotropic perimeter}]
\label{prop:area_minimizing_cones}
Suppose that $\Phi: \R^{n+1} \to \oic$ is cylindrical over
$\phisec.$
Let $H_1, H_2 \subset \R^{n+1}$ be two half-spaces, with outer
unit normals $\nu_1,\nu_2\in\S^{n}$ respectively. Suppose that
\begin{equation}\label{eq:partial_H_i}
\{0\} \in \p H_1\cap \p H_2 \subset \{t=0\},
\end{equation}
and that
\begin{itemize}
 \item[(a)] $\nu_1\cdot\nu_2\ge0,$  $\nu_2\cdot e_{n+1}\ge \nu_1\cdot e_{n+1}\ge0$;
 \item[(b)] $\arccos(\nu_1\cdot\nu_2)+ \arccos(\nu_2\cdot e_{n+1} ) = \arccos(\nu_1\cdot e_{n+1}).$
\end{itemize}
Then the cones $E:= H_1\cap H_2$ and $F := H_1\cup H_2$
are minimizers of $P_\Phi$ in  $\R^{n+1}.$
\end{proposition}

Before proving the proposition, some comments are in order. 
Our assumptions on $H_1$ and $H_2$ exclude, in particular,
 that $E$ is a ``roof-like''
cone (as the one depicted in Figure \ref{Fig:rooflike_cone}). More specifically, in case $\nu_1
\neq \nu_2,$
the inclusion
$\partial H_1 \cap \partial H_2 \subset \{t=0\}$ in \eqref{eq:partial_H_i}
implies that the orthogonal
complement to $\{t=0\}$ is contained in the span of the orthogonal
complements of $\partial H_i,$ {\it i.e.}
$$
e_{n+1} \in {\rm span}(\nu_1,\nu_2).
$$
\begin{figure}
\includegraphics[scale=0.4]{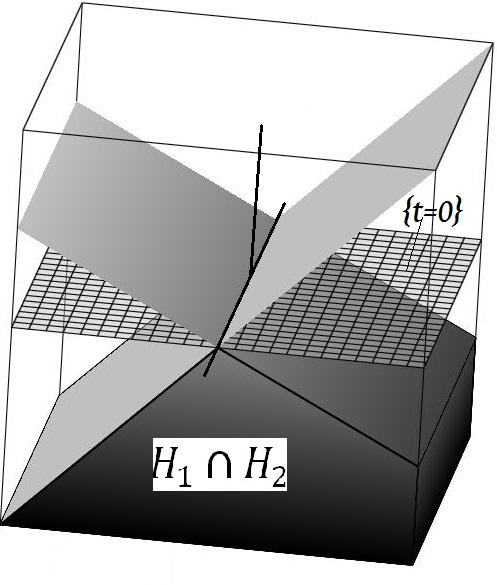}\qquad
\includegraphics[scale=0.4]{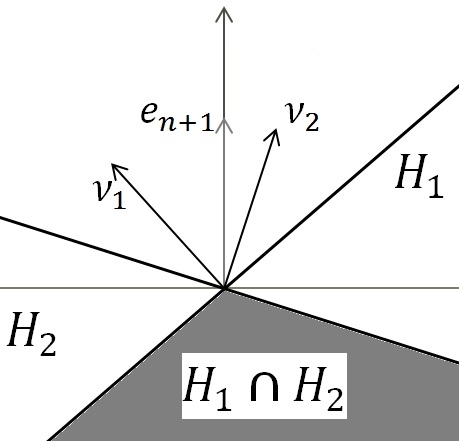}
\caption{``Roof'' like cone (left) and its section (right) along $(\p H_1\cap\p H_2)^\perp.$}
\label{Fig:rooflike_cone}
\end{figure}

Next, assumption (a) implies that $\nu_1$ and $\nu_2$ lie
``on the same side'' with respect to $e_{n+1},$  while assumption
(b) implies that $\nu_2$ lies between $\nu_1$ and $e_{n+1}$ (a condition
not satisfied in  Figure \ref{Fig:rooflike_cone}, and 
satisfied in Figure 
\ref{Fig:minCones}). 
We shall see in Example \ref{ex:non_minimal_cones}
that,  if condition (b) is not satisfied, then $E$ and $F$ need  not
be minimizers.
\begin{proof}
{}For 
$i=1,2,$
define
$$\lambda_i:=
\begin{cases}
\dfrac{\sqrt{1- (\nu_i\cdot e_{n+1})^2}}{\nu_i\cdot e_{n+1}} 
=\tan\alpha_i & {\rm if}~  \nu_i\cdot e_{n+1}\ne 0,
\\
+\infty & {\rm if}~ \nu_i\cdot e_{n+1}=0,
\end{cases}
$$ 
see Figure \ref{Fig:minCones}, left. By (a)
we have
$\lambda_2\le \lambda_1.$ 
Let $\widehat{\nu} \in\S^{n-1} \subset \{t=0\}$ 
be a unit normal 
to $\p H_1\cap\p H_2$ which, 
 according to (b), can be chosen so that
$
(\widehat{\nu},0)\cdot \nu_i\le 0,\quad i=1,2. 
$

If $\lambda_2=+\infty,$ then by conditions (a) 
and (b) we have $H_1=H_2=H,$
where $H$ is the half-space whose outer unit normal 
is $-(\widehat{\nu},0).$ By Example \ref{ex:half-spaces}
it follows that $H=E=F$  is a minimizer of $P_\Phi.$

Assume that $\lambda_2\le\lambda_1<+\infty.$ 
Define
$$
f(\sigma) := \begin{cases}
\lambda_2\sigma,&\sigma \ge 0\\
\lambda_1 \sigma,&\sigma <0
\end{cases},\qquad 
g(\sigma) := \begin{cases}
\lambda_1\sigma,&\sigma \ge 0\\
\lambda_2 \sigma,&\sigma <0
\end{cases}.
$$
Then $E=\sg(u),$ $F=\sg(v),$ where
$u(x) := f(x\cdot \widehat{\nu}),$ $v(x) := g(x\cdot \widehat{\nu}),$ $x\in\R^n.$
Since $f,g$ are {\it monotone}, by Proposition 
\ref{prop:cartesian_min_cyl}(f) we have $u,v\in \cM_{\Phi^o}(\R^n).$
Since $\Phi^o$ is partially monotone (recall Example \ref{ex:Part_mon}),
Theorem \ref{teo:Subgraphs}(b) yields that $E$ and $F$ are minimizers of $P_\Phi$
in $\R^{n+1}.$
\begin{figure}[h]\label{Fig:minCones}
\includegraphics[scale=0.4]{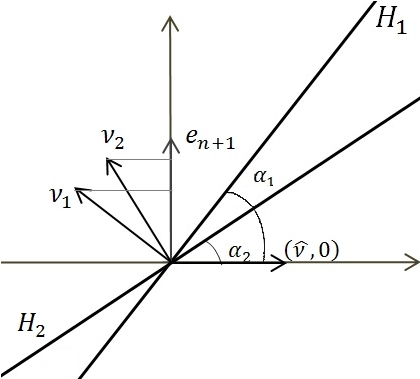}\hspace{2cm}
\includegraphics[scale=0.4]{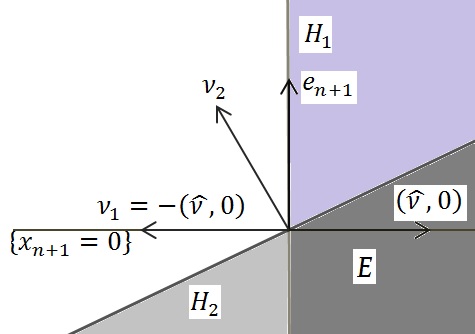}
\caption{Sections of cones when $\lambda_1<+\infty$ and $\lambda_1=+\infty.$} 
\end{figure}
Now, assume that $0\le \lambda_2<\lambda_1=+\infty.$
Then $\nu_1 = -(\widehat{\nu},0).$
We prove that $F$ is a minimizer of  $P_\Phi$ in $\R^{n+1}$
(the proof for $E$ being similar).
It is enough to show minimality of $F$ inside every
strip $S_m=\R^n\times(-m,m),$ $m>0.$
Define $h:\R\to\R$ as
$h(\sigma) := m\chi_{(0,+\infty)}(\sigma)$ 
if $\lambda_2=0$ and
$$
h(\sigma) := \begin{cases}
-m & {\rm if}~ \sigma<-\frac{m}{\lambda_2},
\\
\lambda_2 \sigma & {\rm if}~ -\frac{m}{\lambda_2}\le \sigma <0,
\\
m & {\rm if}~ \sigma \ge 0
\end{cases}
$$
if $\lambda_2>0.$ Let $w(x):=h(x\cdot \widehat{\nu}),$ $x\in\R^n.$
As before, the subgraph $\sg(w)$
of $w$ is a minimizer of $P_\Phi$  in $\R^{n+1}.$
Since $\sg(w)\cap S_m = F\cap S_m,$
it follows that $F$ is a minimizer of $P_\Phi$ in $S_m.$
\end{proof}

\begin{remark}
It is not difficult to see that in Proposition \ref{prop:area_minimizing_cones} 
the assumption $\p H_1\cap \p H_2\subset \{t=0\}$
is in general not necessary. 
Indeed, assume  $n=2,$   $\Phi^o(\xi_1^*,\xi_2^*,\xi_3^*)=|\xi_1^*|+
|\xi_2^*|+|\xi_3^*|$ and  $H_i,$ $i=1,2$ are half-spaces with outer unit normals $
\nu_1=(\frac12,\frac{1}{\sqrt2},\frac12),$ 
$\nu_2= (\frac{1}{\sqrt{10}},\frac{2}{\sqrt{5}},\frac{1}{\sqrt{10}})$ respectively.
Then both $H_1\cap H_2$ and $H_1\cup H_2$ are minimizers of $P_\Phi$ in $\R^3.$
Indeed, for the  Euclidean isometry 
$U(x,t):=(x_1,t,x_2),$ one sees that $UH_1$
and $UH_2$ satisfy the assumptions of Proposition \ref{prop:area_minimizing_cones}, 
hence $UH_1\cap UH_2$ and
$UH_1\cup UH_2$ are minimizers of $P_\Phi.$ Since $\Phi\circ U=\Phi,$ the thesis
follows.
\end{remark}

\begin{example}[\textbf{Non minimal cones}]\label{ex:non_minimal_cones}
Let $n=1,$ $\Phi^o(\xi_1^*,\xi_2^*)= |\xi_1^*|+|\xi_2^*|$  
and $H_1$ and $H_2$ be half-planes of $\R^2$ with outer 
unit normals $\nu_1,\nu_2\in\S^1$ such that
\begin{itemize}
 \item[(a)] $\p H_1\cap\p H_2 = \{0\};$
 \item[(b)] $\nu_1\cdot\nu_2\ge0,$  $\nu_2\cdot e_2\ge \nu_1\cdot e_2\ge0,$ and
 if $\nu_2\cdot e_2 =1$ then $0<\nu_1\cdot e_2<1;$
 \item[(d)] $\arccos(\nu_1\cdot\nu_2)= 
\arccos(\nu_1\cdot e_2)+
\arccos(e_2\cdot \nu_2 ).$ 
\end{itemize}
Then the cones $E:=H_1\cap H_2$ and $F:=H_1\cup H_2$
are not minimizers of $P_\Phi.$ 
Let us  prove the assertion for $E,$ the statement for $F$ being similar.
The lines $\p H_1,$ $\p H_2$ and $\{t=-1\}$
compose a nondegenerate 
triangle $T\subset E$ with sides $a_1,a_2, b>0,$ $b$ the horizontal side.
For any 
$A\in \cK(\R^2)$  with $T\strictlyincluded A$
we have
\begin{align*}
P_\Phi(E,A) - P_\Phi(E\setminus T,A) = a_1\Phi^o(\nu_1)+a_2\Phi^o(\nu_2) - b\Phi^o(e_2)\ge
a_1+a_2-b>0,
\end{align*}
since $\Phi^o(\nu)\ge 1$ for all $\nu\in \S^1.$ 
Hence,  $E$ is not a minimizer of $P_\Phi.$
\end{example}

We shall need the following relevant
result (see for instance 
\cite[Theorem 17.3]{Gi:84} and references therein).

\begin{theorem}\label{teo:giusti} 
Let $\widehat E$ be a minimizer of the Euclidean perimeter
in $\Rn.$ Then either $n \geq 8$ or $\partial \widehat E$ is a hyperplane.
\end{theorem}

Our classification result of minimizers of $\mathcal G_{\Phi^o}$ reads as follows:

\begin{theorem}[\textbf{Entire cartesian minimizers}]\label{teo:Euc}
Suppose that 
$\Phi: \R^{n+1}\to \oic$ is cylindrical over $\phisec.$
Assume one of the two following alternatives:
\begin{itemize}
\item[(a)] $1 \leq n\le 7$ and $\phisec$ is Euclidean;
\item[(b)] $n =2$ and $\phisec^o$ is strictly convex.
\end{itemize}
If $u$ is a minimizer of $\cG$ in $\Rn$
then there exists $\direction\in\S^{n-1}$ and a monotone function $f:\R\to\R$
such that 
\begin{equation}\label{eq:one_variable}
u(x)=f(x\cdot\direction), \qquad  x \in \Rn. 
\end{equation}
\end{theorem}

\begin{remark}
If $\phisec$ is a noneuclidean smooth and uniformly convex norm,
the conclusion  of Theorem \ref{teo:Euc} under assumption (a) does not
necessarily hold. For example, if $n=4$ and $K$ is the cone over
the Clifford torus \cite{Mo:} -- a minimizer of $P_\phisec$ in $\R^4$ for some uniformly
convex smooth norm $\phisec$ -- then by Proposition 
\ref{prop:cartesian_min_cyl}(d), 
$u=\chi_K$ is a minimizer of $\cG$ in $\R^4$ which cannot be represented as 
in \eqref{eq:one_variable}.
We don't know if there are counterexamples also for $n=3$.
\end{remark}

\begin{proof}
Let $u \in \mathcal M_{\Phi^o}(\Rn).$ 
By Corollary \ref{cor:locally finiteness}, $u \in L^\infty_{\rm loc}(\R^n).$
Let 
$$
c_0 := \essinf\limits_{x\in\Rn} u(x) \in [-\infty,+\infty),\qquad
c_1 := \esssup\limits_{x\in\R^n} u(x) \in (-\infty, +\infty].
$$
If $c_0=c_1,$ then  $u \equiv c_0 $ a.e. on $\R^n.$ In this case
$\direction \in\S^{n-1}$ can be chosen arbitrarily and 
$f\equiv 
c_0.$

Assume that $-\infty\le c_0<c_1\le+\infty.$ 
Given $\lambda 
\in \R,$ 
 Proposition \ref{prop:cartesian_min_cyl}(c) implies that 
$\{u > \level\}$ is a minimizer of $P_\phisec$  in $\R^n.$ 
We claim that 
either $\p^*\{u>\level\}$ is a hyperplane
or $\p^*\{u>\level\}=\emptyset.$
Indeed, if  $n=1$ the claim is trivial. If
 $n=2$ and 
$\phisec$  strictly convex, the claim 
follows from \cite[Theorem 3.11]{NoPa:05}. When 
 $3\le n\le 7$ 
and $\phisec$ is  Euclidean, the claim 
 is implied by Theorem \ref{teo:giusti}.
Hence for any $\level\in (c_0,c_1)$ there exist
$\direction_\level\in\S^{n-1}$ and $a_\level\in\R$ such that
\begin{equation}\label{halfspace}
\{u>\level\} =\{x\in\R^n:\,\, x\cdot\direction_\level <  a_\level\}.
\end{equation}
In addition, these hyperplanes cannot intersect transversely,  
hence there exists $\direction\in\S^{n-1}$ such that 
$\direction_\level=\direction$ for all $\level\in (c_0,c_1).$
Since 
the function $\level
\in (c_0,c_1)
\mapsto a_\level$ is monotone, it 
 remains to construct the function $f.$
We may assume that $\level\mapsto a_\level$ is nonincreasing,
the nondecreasing case being similar.
Extend $a_\lambda$ to $\R\setminus[c_0,c_1]$ setting $a_\level:=+\infty$ 
for $\level<c_0$ if $c_0 \in \R,$  and $a_\level:=-\infty$ for $\level>c_1$
if $c_1 \in \R.$ Then, we 
define 
$$f(\sigma):=\sup\left\{\level:\,\, \sigma<a_\level \right\},\quad \sigma\in\R,
$$
which is 
 nonincreasing. 
Note that $f$ is real valued.
Indeed, if $f(\sigma)=-\infty$ for some $\sigma\in\R,$ 
then $\sigma\ge a_\level$ for all $\level\in\R$ which is impossible 
since $a_\level\to+\infty$ as $\level\to-\infty.$
Similarly, $f(\sigma) <+\infty$ for any $\sigma \in \R.$

Set $v(x):=f(x\cdot\direction).$ By construction, we have 
$\{v > \lambda\} = \{u > \lambda\}$ for a.e. $\lambda \in \R.$
It is easy to check that if $w\in L_\loc^1(\R^n)$ then for a.e. 
$x\in\R^n$ one has 
$$
w(x)= \int_0^{+\infty}\chi_{\{w>\level\}}(x)d\level+\int_{-\infty}^0(1-\chi_{\{w>\level\}}(x))d\level,
$$
hence 
$u=v$ almost everywhere  on $\R^n.$
\end{proof}

\begin{remark}\label{rem:Bernstein}
It seems not easy to generalize Theorem 
\ref{teo:Euc} to noneuclidean $\phisec$ 
(for some $n \in \{3,\dots,7\}$)\footnote{
If $\phisec$ is $\mathcal C^\infty$-uniformly convex norm and $n = 3,$ then
$\{u \geq \level\}$ is smooth \cite[Theorem II.7]{AlScYa:77}. 
}, since our argument
was   based on Theorem \ref{teo:giusti}.
\end{remark}

\begin{remark}\label{rem:opt_Euc}
Assumption (a)  of Theorem \ref{teo:Euc} is optimal in the sense that 
if $n\ge8$ there exist minimizers of $\cG$ on $\R^n$ which 
cannot be written as in \eqref{eq:one_variable}. 
Indeed, let $C\subset\R^8$ be the Simons cone minimizing the Euclidean perimeter
\cite[Theorem A]{BDG:69}. By Proposition 
\ref{prop:cartesian_min_cyl}(d) 
$u=\chi_C \in \mathcal M_{\Phi^o}(\Rn),$ 
however $u$ does not admit the  representation \eqref{eq:one_variable}.
\end{remark}

{}From Theorem \ref{teo:Euc} and Proposition \ref{prop:cartesian_min_cyl} (f) we deduce
the following result. 

\begin{corollary}[\textbf{Composition of linear and monotone functions}]\label{cor:char.min.cylindrical}
Under the assumptions of Theorem \ref{teo:Euc}, $u$ is a minimizer of 
$\mathcal G_{\Phi^o}$ in $\Rn$ if and only if there exists $\direction \in \mathbb S^{n-1}$ 
and a monotone function $f: \R \to \R$ such that $u(x) = f(x\cdot \direction)$ 
for any $x \in \Rn.$
\end{corollary}

\section{Lipschitz regularity of cartesian minimizers for cylindrical norms}\label{sec: LipRegCarMin}
We recall from \cite[Theorem 3.12]{NoPa:02} that if $n=2$ and if
$\partial B_\phisec$ either does not contain 
segments, or it is locally a graph in a neighborhood of
its segments, then the graph of a minimizer of $\mathcal G_{\Phi^o}$
in $\R^2$ is locally Lipschitz. 
On the other hand, an example in \cite[Sect. 4]{NoPa:02} shows
that such a regularity result
cannot be expected for a general anisotropy.
More precisely, 
for  $\Phi^o$ cylindrical as in \eqref{eq:phi_dual_conical} 
with $\phisec^o(\hsecdual)=|\xi_1^*|+|\xi_2^*|,$
that example exhibits  a function $u\in\mathcal M_{\Phi^o}(\R^2)$
such that the set of points where the boundary of $\sg(u)$  is not locally the graph
of a Lipschitz function has positive $\mathcal H^2$-measure.  
We look for sufficient conditions on $\phisec$ which exclude such pathological
example.

Let us start with a regularity
property of cartesian minimizers of $\mathcal G_{\Phi^o}$ for {\it cylindrical norms over the 
Euclidean norm}, namely for 
$$
\Phi(\widehat \xi, \xi_{n+1}) 
 = \max(\vert \widehat \xi\vert, \vert \xi_{n+1}\vert),
$$
which is exactly the case of  the total variation functional.

We need the following regularity result, a special case of 
\cite[Theorem 1]{Ta:82}.

\begin{theorem}\label{teo:tama}
Let $\{\widehat E_h\}$ be a sequence of minimizers of the
Euclidean perimeter in 
$\widehat \Omega$ locally
converging to a set $\widehat E$ in $\widehat \Omega,$ and let
$x_h \in \partial \widehat E_h$ be such that $\displaystyle \lim_{h \to +\infty}
x_h =  x \in 
\partial^* \widehat E.$ Then there exists $\overline h\in \N$
such that  $x_h\in \p^* \widehat E_h$ for any $h \in \N,$ $h \geq \overline h,$
and   $\displaystyle \lim_{h \to +\infty} \nu_{\widehat E_h}(x_h) = 
\nu_{\widehat E}(x).$
\end{theorem}

\begin{theorem}[\textbf{Local Lipschitz regularity}]\label{teo: reg_in_all_dim}
Suppose that 
 $u\in BV_{{\rm loc}}(\opensetRn)$ is a minimizer of the total variation functional 
$$
TV(v,\opensetRn) := \int_{\opensetRn}
\vert Dv\vert, \qquad v \in BV_{\rm loc}(\opensetRn).
$$
Then there exists a closed set  $\Sigma(u) \subseteq \partial \sg(u)$  
of Hausdorff dimension at most $n-7$, 
with $\Sigma(u) = \emptyset$ if $n \leq 7,$ 
such that $\partial \sg(u) \setminus \Sigma(u)$ 
is locally Lipschitz.
\end{theorem}

\begin{proof}
By Proposition \ref{prop:cartesian_min_cyl}(c)
the sets $\{u >\level\}$ and $\{u\ge \level\}$ are minimizers
of the Euclidean perimeter in $\opensetRn$ for every $\level\in\R.$
Let $\lambda\in\R$ be such 
that $\p\{u>\level\}$ (resp. $\{u\ge \level\}$) is nonempty.  
{}From  classical regularity results (see for 
instance \cite[Theorem 11.8]{Gi:84} and references therein)
it follows that $\p\{u>\lambda\}$ (resp. $\p\{u\ge \level\}$) {\it is of class 
$C^\infty$} out of a closed set $\Sigma_\lambda^>(u)$ (resp. $\Sigma_\lambda^\ge(u)$) 
of Hausdorff dimension at most $n-8.$
Define
$$
\Sigma(u) :=\{(x,\lambda)\in \p\sg(u):\,\,x\in\Sigma_\level^>(u)\text{~or} ~ x\in \Sigma_\level^\ge(u)\},
$$  
so that 
$\Sigma(u)$ has dimension at most $n-7.$ {}From Theorem \ref{teo:tama}
it follows that $\Sigma(u)$ is closed.

Fix
\begin{equation}\label{eq:x_lambda}
(x,\level)\in\p{\sg}(u)\setminus \Sigma(u).
\end{equation}
One of the following  three (not necessarily mutually exclusive)  cases holds:
\begin{itemize}
 \item[a)] $x\in {\rm int}(\{u=\lambda\})$;
 \item[b)] $x\in\p\{u>\level\};$
 \item[c)] $x\in\p\{u\ge\level\}.$
\end{itemize}

In case a) 
$u$ is locally
constant around $x,$ thus, the assertion is immediate.

Assume b). 
We prove that there exists $r_x>0$ such that 
$\p\{u>\mu\}$ is a graph in direction $\nu_{\{u > \lambda\}}(x)$  
for every $\mu\in\R$ such that $\p\{u>\mu\} \cap B_{r_x}(x)\ne\emptyset.$ 
Indeed, otherwise there would exist $\epsilon>0$ and an infinitesimal
sequence $(r_h) \subset 
(0,+\infty),$ and sequences 
$(\mu_h)\subset \R,$ $(x_h)$ with 
$x_h\in \p^*\{u>\mu_h\}\cap B_{r_h}(x)$ and 
\begin{equation}\label{equ12}
|\nu_{\{u>\level\}}(x) - \nu_{\{u>\mu_h\}}(x_h)|\ge\epsilon
\qquad \forall h\in\N. 
\end{equation}
By Corollary \ref{cor:locally finiteness} $u$ is locally bounded, thus 
$(\mu_h)$ is bounded and we can extract a (not relabelled) 
subsequence converging 
to some $\overline{\lambda}\in\R.$
There is no loss of generality in assuming $(\mu_h)$ nondecreasing.
Then $\{u>\mu_h\}\to \{u\geq \overline{\lambda}\}$ 
in $L_\loc^1(\opensetRn)$ as $h \to +\infty.$ 
By \eqref{eq:x_lambda} we have 
$x\in \p^*\{u\geq \overline{\lambda}\},$ hence from Theorem \ref{teo:tama}
it follows
\begin{equation}\label{equ21}
\nu_{\{u>\mu_h\}}(x_h)\to \nu_{\{u\geq \overline{\lambda}\}}(x)\quad {\rm as}~ h\to+\infty.
\end{equation}
Clearly, either $\{u\geq \overline{\lambda}\}\subseteq \{u>\level\}$ or
$\{u\geq \overline{\lambda}\}\supseteq \{u>\level\}.$ Since  
$$x\in \p^*\{u>\level\}\cap \p^*\{u\geq \overline{\lambda}\}$$
and
$\p\{u\geq \overline{\lambda}\}$ and $\p\{u>\level\}$ are smooth around $x,$
necessarily
$$
\nu_{\{u>\level\}}(x)=\nu_{\{u\ge \overline{\lambda}\}}(x).
$$
But then from \eqref{equ12} and \eqref{equ21} we get 
$$
\epsilon \le |\nu_{\{u>\mu_h\}}(x_h) - \nu_{\{u>\level\}}(x)|\to 0 \quad {\rm as}~ h\to+\infty,
$$
a contradiction.

Thus, for every $x\in\p^*\{u>\level\}$ there exist $r_x>0$ and 
$\epsilon\in(0,1)$ such that for any $\mu \in (\level-r_x,\level+r_x)$ and 
$y\in \p^*\{u>\mu\}\cap B_{r_x}(x)$ one has 
$$
\nu_{\{u>\level\}}(x)\cdot \nu_{\{u>\mu\}}(y)\ge \epsilon.
$$

Notice that for any $(y,\mu)\in\p^*{\sg}(u) \setminus \Sigma(u)$ 
one has that 
$$
{\rm either} \quad \nu_{\sg(u)}(y,\mu)=\dfrac{(\nu_{\{u>\mu\}}(x),\sigma)}{\sqrt{1+\sigma^2}} ~
{\rm for ~some}~ \sigma \ge0, \quad
\text{or}\quad \nu_{\sg(u)} (y,\mu)=e_{n+1}.
$$
We want to prove 
that there exist $\rho>0,$ $\eta\in\S^{n}$ and $c\in(0,1)$ such that  
$\cH^{n}$-every $(y,\mu)\in \p\,\sg(u)\cap B_{\rho}(x,\level)$ there holds 
\begin{equation}\label{lipscitz}
\eta\cdot \nu_{\sg(u)}(y,\mu)\ge c,
\end{equation}
so that 
 \cite[Lemma 3.10]{NoPa:02}
implies that $\p\,\sg(u)\cap B_{\rho}(x,\level)$ is a Lipschitz graph 
in the direction $\eta$
with Lipschitz constant $L=\sqrt{1/c^2-1}.$

Set
$$
\rho=r_x,\quad \eta=\frac{1}{\sqrt2}\,(\nu_{\{u>\level\}}(x),1).
$$
Then for any $(y,\mu)\in\p^*\sg(u)\cap B_\rho(x,\level)$ we have 
\begin{equation}\label{yvert}
 \nu_{\sg(u)}(y,\mu)\cdot \eta = \frac{1}{\sqrt2},
\end{equation}
if 
$\nu_{\sg(u)}(y,\mu)=e_{n+1},$
and 
\begin{equation}\label{yonbound}
\nu_{\sg(u)}(y,\mu)\cdot \eta = \frac{\nu_{\{u>\mu\}}(y) \cdot \nu_{\{u>\level\}}(x) +s 
}{\sqrt{2}\sqrt{1+s^2}}\ge \frac{\epsilon+s}{\sqrt{2}\sqrt{1+s^2}}\ge 
\dfrac{\epsilon}{\sqrt2}, 
\end{equation}
if $y\in \p\{u>s\}$ (here we use $\frac{a+s}{\sqrt{1+s^2}}\ge a$ for any 
$a\in(0,1)$ and $s\ge0$). Formulas
\eqref{yvert} and \eqref{yonbound} imply \eqref{lipscitz} with $c=\epsilon/\sqrt2.$

Finally, case c) can be treated as case b).

\end{proof}

\begin{remark}\label{rem:opt_Bernstein}
The assertion of Theorem \ref{teo: reg_in_all_dim} cannot be improved:
if $n\ge 8,$ there exists a minimizer $u$ of $\cG$  
such that 
the points where $\p\sg(u)$ is  not locally Lipschitz
have positive $(n-7)$-dimensional Hausdorff measure. 
For the Simons cone in $\R^8$ (and with the Euclidean norm), 
the graph of $u=\chi_C$ cannot be represented as the graph of a 
Lipschitz function in a neighborhood of the origin.
\end{remark}

Theorem \ref{teo: reg_in_all_dim} can be generalized 
as follows.

\begin{theorem}\label{teo:Al}
Suppose that  $\Phi: \R^{n+1} \to \oic$ is cylindrical
over $\phisec$  with
$$
\phisec^2\in C^3(\R^{n})\quad{\rm is~ uniformly~ convex}.
$$
If $u$ is a minimizer of $\cG$ in $\opensetRn,$
then 
$\p\sg(u) \setminus \Sigma(u)$
is locally Lipschitz,
where $\Sigma(u) \subseteq \partial \sg(u)$ is a closed set
of Hausdorff dimension at most $n-2$ if $n>3,$ and $\Sigma(u) = \emptyset$ if $n=2,3.$
\end{theorem}

\begin{proof}
The proof is the same as in Theorem \ref{teo: reg_in_all_dim}, 
using 
\cite[Theorems II.7]{AlScYa:77} in place of \cite{Gi:84},
and \cite[Theorem 4.5]{Ne:2015} in place of \cite{Ta:82}.
\end{proof}

\begin{remark}
In 
\cite{NoPa:02}
it is proven that
if $n=2,$ 
$B_\phisec$ is not a quadrilateral, and 
$u$ is a minimizer of $\cG$ in $\opensetRn,$
then the graph of $u$ is locally Lipschitz around any point of $\p\sg(u).$
\end{remark}

\begin{remark}
Using the regularity result in \cite[Theorem II.8]{AlScYa:77}, under
the assumption that $\phisec$ is uniformly convex, smooth and sufficiently 
close to the Euclidean norm, 
one can improve Theorem \ref{teo:Al} by showing that 
$\Sigma(u)$ has Hausdorff dimension at most $n-5.$
\end{remark}

\appendix \section{}

\subsection{A Fubini-type theorem}

%
\begin{proposition}\label{prop:Miranda_Fubini}
Let $E\in BV_\loc(\opensetRnplusone).$
Then for any $A\in\cK(\opensetRnplusone)$ 
\begin{align}\label{eq:hor_per}
\int_{A\cap \p^*E} \Phi^o(\widehat\nu_E,0)d\cH^n = 
\int_\R dt\int_{A_t\cap \p^*E_t} \Phi^o(\nu_{E_t},0)d\cH^{n-1},
\end{align}
\begin{align}\label{eq:ver_per}
\int_{A\cap \p^*E} \Phi^o(0, (\nu_E)_t)d\cH^n = 
\int_{\R^n} dx\int_{A_x\cap \p^*E_x} \Phi^o(0,1)d\cH^{0}.
\end{align}
where $E_t$ and $E_x$ are defined as \eqref{eq:Sections},  
$\nu_{E_t}$ is a outer unit normal to $\p^*E_t$ and 
$\nu_{E_x}$ is a outer unit normal to $\p^*E_x.$ 
\end{proposition}

\begin{proof}
Let us prove \eqref{eq:hor_per}.
Notice that by \cite[Theorem 18.11]{Maggi} for a.e. $t\in\R$ 
$$
\cH^{n-1}(\p^*E_t \Delta (\p^*E)_t) =0, \quad \widehat\nu_E\ne0,\quad 
\nu_{E_t}=\frac{\widehat\nu_E}{|\widehat\nu_E|}.$$
We can use the coarea formula \cite[Theorem 2.93]{AFP:00} with the function 
$f:\R^{n+1}\to\R,$ $f(x,t)=t.$
Then $\nabla f=e_{n+1}$ and its orthogonal projection $\nabla^E f$ on the 
approximate tangent space to $\p^*E$ is 
$\nabla^E f= e_{n+1} - (e_{n+1}\cdot \nu_E)\nu_E.$
Thus,
\begin{align*}
& \int_\R dt\int_{A_t\cap \p^*E_t} \Phi^o(\nu_{E_t},0)d\cH^{n-1} =
\int_\R dt\int_{(A\cap\p^*E)\cap \{f=t\}} \Phi^o(\nu_{E_t},0)d\cH^{n-1} 
\\
=& \int_{A\cap\p^*E} \Phi^o(\nu_{E_t},0)|e_{n+1} - (e_{n+1}\cdot \nu_E)\nu_E|\,d\cH^{n-1} 
= \int_{A\cap\p^*E} \Phi^o(\nu_{E_t},0)\sqrt{1 - |(\nu_E)_t|^2} \,d\cH^{n-1} \\
=& \int_{A\cap\p^*E} \Phi^o(\nu_{E_t},0)|\widehat\nu_E| \,d\cH^{n-1} 
= \int_{A\cap\p^*E} \Phi^o(\widehat\nu_E,0) \,d\cH^{n-1}.
\end{align*}
Now, \eqref{eq:ver_per} follows from \eqref{perii} 
and \cite[Theorem 3.3]{MM1964}:
\begin{align*}
\int_{A\cap \p^*E} \Phi^o(0, (\nu_E)_t)d\cH^n = & 
\Phi^o(0,1) \int_{A\cap \p^*E} |(\nu_E)_t|d\cH^n 
= \Phi^o(0,1) \int_A|D_t\chi_E|\\
=&\Phi^o(0,1)  \int_{\R^n} dx\int_{A_x\cap \p^*E_x} d\cH^{0}. 
\end{align*}

\end{proof}

\begin{remark}\label{rem:miranda} 
Let 
$\Phi:\R^{n+1}\to\oic$ 
be a norm. For notational simplicity set
$\phisec_1:=\Phi_{ \vert_{ \{\xi_{n+1} = 0\} } }^{},$ $\phisec_2:=\Phi_{\vert_{\{\widehat{\xi} = 0\}}}^{}.$
For $f\in BV_{\loc}(\opensetRnplusone)$ and $A\in\cK(\opensetRnplusone)$ we define
\begin{align*}
\int_{A} \phisec_1^o(D_xf)=\sup\Big\{ \int_A f(x,t) &\sum\limits_{i=1}^n
\frac{\partial \eta_i(x,t)}{\partial x_i} \,dxdt:\,
 \eta\in C_c^1(A; B_{\phisec_1})\Big\},\\
\int_{A} \phisec_2^o(D_tf)=\sup\Big\{ \int_A f(x,s) &D_t\eta(x,s) \,dxds:\,
 \eta\in C_c^1(A),\,\,\phisec_2(\eta)\le 1\Big\}.
\end{align*}
With this notation \eqref{eq:hor_per} and \eqref{eq:ver_per} can be rewritten 
respectively as\footnote{Following \cite[Theorem 3.3]{MM1964} one can prove 
a more general statement, namely, 
if $f\in BV_{\rm loc}(A),$ then
\begin{align*}
\int_A \phisec_1^o(D_xf) =\int_{\R} dt
\int_{A_t} \phisec_1^o(D_x (f\big|_{A_t})),\qquad 
\int_A \phisec_2^o(D_tf) =\int_{\R^n} dx
\int_{A_x} \phisec_2^o(D_t (f\big|_{A_x})).
\end{align*}
}
\begin{align*}
\int_{A} \phisec_1^o(D_x\chi_E) = &
\int_\R dt\int_{A_t} \phisec_1^o(D_x\chi_{E_t}),\\
\int_{A} \phisec_2^o(D_t\chi_E) = &
\int_{\R^n} dx\int_{A_x} \phisec_2^o(D_t\chi_{E_x}). 
\end{align*} 
\end{remark}

\subsection{Norms with generalized graph property}\label{subsec:norms_with_generalized_graph_property}

\begin{lemma}\label{lem:AdjointPropertyG}
$\partial B_\Phi$
is a generalized graph in the vertical direction
if and only if $\partial B_{\Phi^o}$
is a generalized graph in the vertical direction.
\end{lemma}

\begin{proof}
Suppose that  $\partial B_\Phi$ 
is a generalized graph in the vertical direction. 
Let $\xi^*=(\hsecdual,0)\in\R^{n+1},$ and take
$\xi = (\widehat \xi, \xi_{n+1})\in\R^{n+1}$ such that $\Phi(\xi)=1$ and
\begin{equation}\label{eq:hat_xi_dot_hat_xi_star}
\hsectangent\cdot\hsecdual =
\Phi^o(\hsecdual,0) = (\hsectangent,\xi_{n+1})\cdot(\hsecdual,0).
\end{equation}
Since $\partial B_\Phi$ is a generalized graph in the vertical direction, 
we have $\Phi(\hsectangent,0)\leq \Phi(\xi)=1.$ Thus,
 by \eqref{duality} and \eqref{eq:hat_xi_dot_hat_xi_star}
we get $\Phi^o(\hsecdual,0)
\le \Phi^o(\hsecdual,\xi_{n+1}^*),$ 
hence $\partial B_{\Phi^o}$ 
is a generalized graph in the vertical direction. The converse conclusion
follows then from the equality $\Phi^{oo} = \Phi.$
\end{proof}

\begin{lemma}\label{lem:CommutRestAdj}
Equality holds in \eqref{restHorPl} if and only if $\partial B_{\Phi}$
is a generalized graph in the vertical direction.
\end{lemma}

\begin{proof}
Set $\phisec := \Phi_{ \vert_{ \{ \xi_{n+1} \} }}^{}.$
Assume that $\partial B_{\Phi}$
is a generalized graph in the vertical direction.
Let  $(\hsecdual,0)\in\R^{n+1}$ and take 
$\xi=(\hsectangent,\xi_{n+1})\in\R^{n+1}$ such that
$\Phi(\xi)=1$ and \eqref{eq:hat_xi_dot_hat_xi_star} holds.
By our assumption on $\p B_\Phi$ it follows that 
$\phisec(\hsectangent) = \Phi(\hsectangent,0)\le  \Phi(\hsectangent,\xi_{n+1})=1,$ 
hence by \eqref{duality}
$$
\Phi^o(\hsecdual,0)\le  \phisec^o(\hsecdual)\phisec(\hsectangent)\le \phisec^o(\hsecdual).
$$
This and \eqref{restHorPl} imply $\phisec^o(\hsecdual) = \Phi^o(\hsecdual,0),$ {\it i.e.}
$
\left(\Phi_{ \vert_{ \{\xi_{n+1} =0\} } }^{} \right)^o =
\Phi_{\vert_{\{\xi_{n+1}^* =0\}}}^o.
$

Now assume that equality in \eqref{restHorPl} holds. 
Take any $\xi=(\hsectangent,\xi_{n+1})\in\R^{n+1}$
and select $\hsecdual\in\R^n$ such that
$\phisec^o(\hsecdual) = \Phi^o(\hsecdual,0)=1$ and 
$\phisec(\hsectangent) = \hsectangent\cdot\hsecdual.$ 
Then by \eqref{duality}
\begin{align*}
\Phi(\hsectangent,0) = &\phisec(\hsectangent) = \hsectangent\cdot\hsecdual
=(\hsectangent,\xi_{n+1})\cdot
(\widehat{\xi^*},0)\le \Phi(\hsectangent,\xi_{n+1}).
\end{align*}
\end{proof}

\subsection{Partially monotone norms}\label{subsec:partially_monotone_norms}

\begin{proposition}[\bf Characterization of partially monotone norms]\label{prop:part_mon}
The norm $\Phi:\R^{n+1}\to\oic$ is partially monotone if and only if 
there exists a positively one-homogeneous convex function 
$\omega:[0,+\infty)\times[0,+\infty)\to\oic$ satisfying 
\begin{equation}\label{Omega_monotone}
\omega(1,0),\omega(0,1)>0,\,\,\omega(s_1,s_2)\le 
\omega(t_1,t_2),\quad 0\le s_i\le t_i,\,\,i=1,2,
\end{equation}
such that 
\begin{equation}\label{ppp}
\Phi(\widehat{\xi},\xi_{n+1}) = \omega(\phisec(\widehat{\xi}), |\xi_{n+1}|),
\end{equation}
where $\phisec=\Phi_{\vert_{ \{\xi_{n+1} = 0\} }}^{}.$ 
\end{proposition}

\begin{proof}
Concerning the  ``if'' part,
one checks that the function $\Phi$ defined as \eqref{ppp} is 
a partially monotone norm on $\R^{n+1}.$ 
Now, let us prove the ``only if'' part. 
Choose any $\widehat\eta\in\R^n$  with $\phisec(\widehat\eta)=1$ and define the function
$\omega:\oic\times\oic\to\oic$ as 
$\omega(s,t) := \Phi(s\widehat\eta,t)$ for $s,t\ge0.$
Since $\Phi$ is convex and positively one-homogeneous, so is $\omega.$ Moreover, 
the relations  $\Phi(\widehat\eta,0)=1,$ $\Phi(0,1)>0$ and partial monotonicity of 
$\Phi$ imply that $\omega$ satisfies \eqref{Omega_monotone}.
Now it remains to prove \eqref{ppp}.
Comparing 
$\xi=(\widehat{\xi},\xi_{n+1})\in\R^{n+1}$ with
$\eta=(\phisec(\widehat{\xi})\widehat\eta,|\xi_{n+1}|)\in\R^{n+1}$ in \eqref{comp_xi_eta}
and using the relation $\Phi(0,\xi_{n+1})=\Phi(0,|\xi_{n+1}|)$ 
and partial monotonicity, we find
$$
\Phi(\widehat\xi,\xi_{n+1}) = \Phi(\phisec(\widehat{\xi})\widehat\eta,|\xi_{n+1}|) = 
\omega(\phisec(\widehat{\xi}),|\xi_{n+1}|).
$$
\end{proof}

Notice that for $\Phi$ as in \eqref{ppp} we have 
\begin{equation}\label{pmnormdual}
\Phi^o(\widehat{\xi^*},\xi_{n+1}^*)=\omega^o(\phisec^o(\widehat{\xi^*}),|\xi_{n+1}^*|),
\end{equation}
where
$\omega^o:[0,+\infty)\times[0,+\infty)\to\oic$ is defined as
\begin{equation}\label{eq:dual_omega}
\omega^o(s_1^*,s_2^*)=\sup\{s_1s_1^*+s_2s_2^*:\,\,s_1,s_2\in\oic,\,\,\omega(s_1,s_2)\le1\},\,\,
s_1^*,s_2^*\in[0,+\infty). 
\end{equation}
Indeed, 
take any $\xi^*=(\widehat{\xi^*},\xi_{n+1}^*)\in(\R^{n+1})^*.$ 
Let $\xi=(\widehat{\xi},\xi_{n+1})\in\R^{n+1}$ be such that 
$\Phi(\xi)=\omega(\phisec(\widehat{\xi}),|\xi_{n+1}|)\le 1$ and
$\xi\cdot \xi^* = \Phi^o(\xi^*).$ Then using \eqref{duality} twice we get
\begin{equation}\label{less}
\begin{aligned}
\Phi^o(\xi^*)= & \widehat{\xi}\cdot\widehat{\xi^*}+ \xi_{n+1}\cdot\xi_{n+1}^*
\le \phisec(\widehat{\xi})\phisec^o(\widehat{\xi^*})+ |\xi_{n+1}|\cdot|\xi_{n+1}^*|\\
\le & \omega(\phisec(\widehat{\xi}),|\xi_{n+1}|) \omega^o(\phisec^o(\widehat{\xi^*}),|\xi_{n+1}^*|)\le 
\omega^o(\phisec^o(\widehat{\xi^*}),|\xi_{n+1}^*|).
\end{aligned} 
\end{equation}

On the other hand, for any $\xi^*\in (\R^{n+1})^*$ there exist $\widehat{\xi}\in\R^{n}$  such that 
$\phisec(\widehat{\xi}) \le1$ and 
$$\widehat{\xi} \cdot\widehat{\xi^*} =\phisec^o(\widehat{\xi^*}).$$
Moreover, by definition of $\omega^o$ one can find $(s_1,s_2)\in\oic\times\oic$ such that 
$\omega(s_1,s_2)\le1$
and  $\omega^o( \phisec^o(\widehat{\xi^*}),|\xi_{n+1}^*| ) =
s_1 \phisec^o(\widehat{\xi^*})+s_2|\xi_{n+1}^*|.$
Using \eqref{Omega_monotone} for $(s_1\phisec(\widehat{\xi}),s_2\sign(\xi_{n+1}^*))$
and $(s_1,s_2)$ one has 
$$\Phi(s_1\widehat{\xi}, s_2\sign(\xi_{n+1}^*) ) =  
\omega(s_1\phisec(\widehat{\xi}),s_2|\sign(\xi_{n+1})| ) \le 
\omega(s_1,s_2)\le1.$$
Thus,
\begin{equation}\label{more}
\begin{aligned}
\omega^o(\phisec^o(\widehat{\xi^*}),|\xi_{n+1}^*|)=&  s_1 \phisec^o(\widehat{\xi^*})+s_2|\xi_{n+1}^*| =
(s_1\widehat{\xi}) \cdot \widehat{\xi^*}+ (s_2\sign(\xi_{n+1}^*))\cdot \xi_{n+1}^*\\
\le& \Phi(s_1\widehat{\xi},s_2\sign(\xi_{n+1}^*) ) \Phi^o(\widehat{\xi^*},\xi_{n+1}^*)\le
\Phi^o(\widehat{\xi^*},\xi_{n+1}^*).
\end{aligned}
\end{equation}
From \eqref{less}-\eqref{more} we get \eqref{pmnormdual}.

\begin{remark}
The norm $\Phi:\R^{n+1}\to\oic$ 
is partially monotone if  and only if $\Phi^o$ is partially monotone.
\end{remark}

We give the proof of the following lemma which is used in 
the proof Theorem \ref{teo:Subgraphs}.

\begin{lemma}\label{lem:PartmonNorm}
Suppose that $\Phi:\R^{n+1}\to\oic$ is a partially monotone norm, 
$E,F\in BV_\loc(\opensetRn\times\R)$ such that for every 
$\bopensetRn\in\cK(\opensetRn)$ 
\begin{equation}\label{cond_to_meas}
\int_{\bopensetRn\times\R} \Phi^o(D_x\chi_E,0)\le 
\int_{\bopensetRn\times\R} \Phi^o(D_x\chi_F,0),\quad
\int_{\bopensetRn\times\R} \Phi^o(0,D_t\chi_E)\le 
\int_{\bopensetRn\times\R} \Phi^o(0,D_t\chi_F).
\end{equation}
Then for any  $\bopensetRn\in\cK(\opensetRn)$ we have 
\begin{equation}\label{fin_comp}
\int_{\bopensetRn\times\R} \Phi^o(D_x\chi_E,D_t\chi_E)\le 
\int_{\bopensetRn\times\R} \Phi^o(D_x\chi_F,D_t\chi_F).
\end{equation}
\end{lemma}

\begin{proof}
We may assume that $\int_{\bopensetRn\times\R} \Phi^o(D_x\chi_F,D_t\chi_F)<+\infty.$ 
Then $\int_{\bopensetRn\times\R} \Phi^o(D_x\chi_E,D_t\chi_E)<+\infty.$
Indeed, since all norms in $\R^{n+1}$ are comparable, there exists $c,C>0$ such that 
$$c\Phi^o(\xi)\le \Phi^o(\widehat{\xi},0)+\Phi^o(0,\xi_{n+1}) \le C\Phi^o(\xi),\quad 
\xi\in\R^{n+1},$$
thus
\begin{align*}
& c\int_{\bopensetRn\times\R} \Phi^o(D_x\chi_E,D_t\chi_E)\le 
\int_{\bopensetRn\times\R} \Phi^o(D_x\chi_E,0)+
\int_{\bopensetRn\times\R} \Phi^o(0,D_t\chi_E)
\\ 
\le& \int_{\bopensetRn\times\R} \Phi^o(D_x\chi_F,0)+
\int_{\bopensetRn\times\R} \Phi^o(0,D_t\chi_F)
\le  C\int_{\bopensetRn\times\R} \Phi^o(D_x\chi_F,D_t\chi_F).
\end{align*}

By definition of $\Phi$-perimeter and Proposition \ref{prop:part_mon}, 
for any $\epsilon>0$ there exists $\eta\in C_c(\bopensetRn\times\R; B_\Phi)$ 
such that $\Phi(\eta)=\omega(\phisec(\widehat{\eta}),|\eta_{n+1}|)\le1$ and 
\begin{equation}\label{fin1_23}
 \int_{\bopensetRn\times\R} \Phi^o(D\chi_E)-\epsilon < - \int_{\bopensetRn\times\R} \div \eta dx =
\int_{\bopensetRn\times\R} \eta\cdot D\chi_E. 
\end{equation}
Then from \eqref{cond_to_meas}, \eqref{ppp}, \eqref{eq:dual_omega} and \eqref{pmnormdual}  we get 
\begin{align*}
& \int_{\bopensetRn\times\R} \eta\cdot D\chi_E  =
\int_{\bopensetRn\times\R}\left( \sum\limits_{j=1}^n \eta_j\cdot D_{x_j}\chi_E + \eta_{n+1}D_t\chi_E\right) 
\\
\le & 
\int_{\bopensetRn\times\R} \left(\phisec(\widehat{\eta})\,d\phisec^o(D_x\chi_E) +|\eta_{n+1}|\,d|D_t\chi_E|\right) 
\le 
\int_{\bopensetRn\times\R} \left(\phisec(\widehat{\eta})\,d\phisec^o(D_x\chi_F) +|\eta_{n+1}|\,d|D_t\chi_F|\right) 
\\ 
\le &\int_{\bopensetRn\times\R} \omega(\phisec(\widehat{\eta}),|\eta_{n+1}|)\, d\omega^o(\phisec^o(D_x\chi_F),|D_t\chi_F|)
\le  \int_{\bopensetRn\times\R} \omega^o(\phisec^o(D_x\chi_F),|D_t\chi_F|))
\\
=&  \int_{\bopensetRn\times\R}  \Phi^o(D\chi_F).
\end{align*}
This inequality, \eqref{fin1_23} and arbitrariness of $\epsilon$ yield \eqref{fin_comp}. 

\end{proof}

\end{document}